\documentclass[11pt,a4paper]{article}
\usepackage[utf8]{inputenc}
\usepackage{amsmath}
\usepackage{amsfonts}
\usepackage[colorlinks,linkcolor=blue,anchorcolor=blue,citecolor=blue]{hyperref}
\usepackage{amssymb}
\usepackage{extarrows}
\usepackage{amsthm}
\usepackage{bm}
\usepackage{tikz}
\usetikzlibrary{arrows.meta}       
\usetikzlibrary{positioning}       
\usetikzlibrary{calc}              
\usetikzlibrary{shapes}  
\usepackage{verbatim}
\usepackage{hyperref}
\usepackage{comment}
\usepackage{cite}
\usepackage{enumerate}
\PassOptionsToPackage{normalem}{ulem}
\usepackage{ulem}
\usepackage[margin=1in]{geometry} % to change the page dimensions
\usepackage{array}
\usepackage{cases}
\usepackage{mathrsfs}
\usepackage{longtable}
\allowdisplaybreaks[4]
\allowdisplaybreaks
\newcolumntype{L}[1]{>{\raggedright\let\newline\\\arraybackslash\hspace{0pt}}m{#1}}
\newcolumntype{C}[1]{>{\centering\let\newline\\\arraybackslash\hspace{0pt}}m{#1}}
\newcolumntype{R}[1]{>{\raggedleft\let\newline\\\arraybackslash\hspace{0pt}}m{#1}}

\usepackage{natbib}
\usepackage{color}
\usepackage{dsfont}
\usepackage{marginnote}
\usepackage{enumitem}
\usepackage{graphicx}
\usepackage{float}

\theoremstyle{plain}
\newtheorem{theorem}{\protect Theorem}[section]
\newtheorem{prop}[theorem]{\protect Proposition}
\newtheorem{definition}[theorem]{\protect Definition}

\newtheorem{lemma}[theorem]{\protect Lemma}
\newtheorem{remark}[theorem]{\protect Remark}
\newtheorem{ass}{\protect Assumption}
\newtheorem{corollary}[theorem]{\protect Corollary}

\def\d{\mathrm{d}}
\def\ie{\mathrm{i.e.}}
\def\as{\mathrm{a.s.}}
\def\RV{\mathrm{r.v.}}

\newcommand{\R}{\mathbb{R}}
\newcommand{\E}{\mathbb{E}}

\newcommand{\T}{\top}
\newcommand{\C}{\mathcal{C}}
\newcommand{\D}{\mathcal{D}}
\newcommand{\F}{\mathcal{F}}
\newcommand{\Fb}{\mathbb{F}}
\newcommand{\Pb}{\mathbb{P}}
\newcommand{\Pc}{\mathcal{P}}

\newcommand{\tr}{\mathrm{tr}}
\newcommand{\Law}{\mathcal{L}}

\newcommand{\Gr}{\mathrm{Gr}}

\newcommand{\Mc}{\mathcal{M}}

\newcommand{\U}{\mathscr{U}}

\newcommand{\Lb}{\mathbb{L}}

\newcommand{\assref}[1]{\hyperref[#1]{Assumption \ref*{#1}}}
\newcommand{\thmref}[1]{\hyperref[#1]{Theorem \ref*{#1}}}
\newcommand{\propref}[1]{\hyperref[#1]{Proposition \ref*{#1}}}
\newcommand{\remref}[1]{\hyperref[#1]{Remark \ref*{#1}}}
\newcommand{\lemref}[1]{\hyperref[#1]{Lemma \ref*{#1}}}
\newcommand{\defref}[1]{\hyperref[#1]{Definition \ref*{#1}}}
\newcommand{\corref}[1]{\hyperref[#1]{Corollary \ref*{#1}}}
\newcommand{\equref}[1]{\hyperref[#1]{(\ref*{#1})}}
\newcommand{\exaref}[1]{\hyperref[#1]{Example \ref*{#1}}}
\title{Mean Field Control with Poissonian Common Noise:\\
	A Pathwise Compactification Approach}

\author{Lijun Bo \thanks{Email: lijunbo@ustc.edu.cn, School of Mathematics and Statistics, Xidian University, Xi'an, 710126, China.}
	\and
	Jingfei Wang \thanks{Email:wjf2104296@mail.ustc.edu.cn, School of Mathematical Sciences, University of Science and Technology of China, Hefei, 230026, China.}
	\and
	Xiaoli Wei \thanks{Email: xiaoli.wei@hit.edu.cn, Institute for Advance Study in Mathematics, Harbin Institute of Technology, Harbin, China.}
	\and
	Xiang Yu \thanks{Email: xiang.yu@polyu.edu.hk, Department of Applied Mathematics, The Hong Kong Polytechnic University, Kowloon, Hong Kong.}
}
\date{\vspace{-0.8cm}}

\begin{document}
	\maketitle

\begin{abstract}
This paper contributes to the compactification approach to study mean-field control problems with Poissonian common noise. To overcome the lack of compactness and continuity issues caused by common noise, we exploit the point process representation of the Poisson random measure with finite intensity and propose a pathwise formulation in a two-step procedure by freezing a sample path of the common noise. In the first step, we establish the existence of optimal relaxed controls in the pathwise formulation as if common noise is absent, but with finite deterministic jumping times. The second step plays the key role in our approach, which is to aggregate the optimal solutions in the pathwise formulation over all sample paths of common noise and show that it yields an optimal solution in the original model. To this end, with the help of concatenation techniques, we first develop a pathwise superposition principle in the model with deterministic jumping times, drawing a relationship between the pathwise relaxed control problem and the pathwise measure-valued control problem. We then further bridge the equivalence among different problem formulations and verify that the constructed solution under aggregation is indeed optimal in the original problem. 

\vspace{0.1in}
    
\vspace{0.1in}
\end{abstract}

	\section{Introduction}
	
	Mean-field control (MFC) features the cooperative interactions when all agents jointly optimize the social optimum in the mean-field regime, which is closely related to mean-field games (MFG) initially introduced by Larsy and Lions \cite{LL2007} and Huang et al.  \cite{Huangetal2006}. Both types of mean field problems have gained remarkable theoretical advancements and vast applications during the past decades. To model more realistic scenarios where external random factors affect all agents simultaneously in the system, the incorporation of common noise in mean field models has caught considerable attention and spurred various recent methodological developments to better understand the dynamics and strategic interactions influenced by common noise.

	Most existing studies on mean field models focus on the Brownian common noise. For MFC problems with Brownian common noise, to name a few, the dynamic programming principle has been established in Pham and Wei \cite{pham2017} under closed-loop controls, in Djete et al. \cite{Djete2022} with a non-Markovian framework and open-loop controls, and in Denkert et al. \cite{Den24} by utilizing the randomization method; the viscosity solution and comparison principle of the HJB equation has been studied in Zhou et al. \cite{Zhou2024}; the limit theory and equivalence between different formulations has been investigated in Djete et al. \cite{Djete}; the time-inconsistent MFC under non-exponential discount and the characterization of the closed-loop time-consistent equilibrium have been discussed in  \cite{LiangZY}. For MFG problems with Brownian common noise, the strong mean field equilibrium (MFE) adapted to the common noise filtration has been established by analyzing the master equation in \cite{Ahuja}, \cite{Card} and \cite{Mou} under some regularity and monotonicity conditions. 

	The probabilistic compactification approach has been another powerful tool to establish the existence of Markovian MFEs in a general mean-field setup since the pioneer study in Lacker \cite{Lacker3}. The idea of compactification originates from the relaxed control formulation in Karoui et al. \cite{Karoui} and Haussmann and Lepeltier \cite{Haussmann} for single agent's control problems. The compactification arguments tackle the law of the controlled system directly and allow for non-unique optimal controls by utilizing a set-valued fixed-point theorem (such as Kakutani’s fixed-point theorem). In MFC and MFG problems without common noise, the compactification method has been generalized and employed in different settings such as MFG with controlled jumps in Benazzoli et al. \cite{Benazzoli}; MFG with absorption in Campi and Fisher \cite{Campi}; MFG with finite states in Cecchin and Fisher \cite{Cecchin}; MFG with singular controls in Fu and Horst \cite{Fu}; MFC with singular control and mixed state-control-law constraints in \cite{BWY1}; MFG of controls with reflected state dynamics in \cite{BWY}. Comparing with these studies without common noise, the consideration of common noise brings significantly more complexities as the limiting environment is described by a stochastic flow of conditional distribution of the population given the common noise. As a key step in the compactification approach, one has to carry out the fixed point argument to the space of measure-valued processes to conclude the consistency condition of MFE, which is however lack of compactness. Another major challenge in the compactification method is the lack of continuity of the conditional law with respect to the joint law when the conditional probability space is not finite. Specifically, the convergence of joint laws $\Law(X_n, Y) \to \Law(X, Y)$ does not imply the convergence of conditional laws $\Law(X_n|Y) \to \Law(X|Y)$, $\Law(Y)$-a.s. when $Y$ takes infinite values. The same technical issues from the lack of compactness and continuity also arise in applying the compactification approach in MFC problems with common noise. To circumvent these technical obstacles, a discretization procedure was initially proposed in Carmona et al. \cite{Carmona1} for MFG with drift control by discretizing the Brownian common noise in space and time, and then taking a suitable limiting argument. As a consequence, the obtained MFE are called the weak MFE as they are not necessarily adapted to the common noise filtration. Later, the same discretization technique of common noise and different levels of generalizations in compactification arguments have been developed in various context such as Barrasso and Touzi \cite{BT22} for MFG with both drift and volatility control, Tangpi and Wang \cite{Tangpi24} for MFG of controls and random entry time, and Burzoni and Campi \cite{Bcampi} for MFG with absorption, all compromised to the existence of weak MFE as in \cite{Carmona1}. In a special and restrictive setting when the interaction incurs via the conditional law given the current value of common noise, Tangpi and Wang \cite{Tangpi25} recently established the existence of strong MFEs using a compactness criterion for Malliavin-differentiable random variables to processes without the step of discretization.

 The goal of the present paper is to contribute new techniques to the compactification approach for MFC when the common noise is depicted by some Poisson random measures. The common Poisson random measures are widely used to capture the impact of unexpected common shock events that affect all participants, such as financial crises, policy interventions, pandemics and natural catastrophes. For instance, Lindskog and McNeil \cite{Lindskog} used Poisson processes to model common windstorms that cause insurance losses across multiple countries. Similarly, Duffie and Garleanu \cite{Duffin} explored the default risk of $N$ participants in a collateral pool, where each obligor's default intensity comprises an idiosyncratic component and a common state process driven by a pure-jump process shared among all obligors. Moreover, Poisson common noise can naturally be applied to systemic risk (c.f. \cite{Fouque}), where the reserves of all interbanks simultaneously under abrupt jumps in response to common shocks, such as major policy announcements. Motivated by these abrupt and discretely occurring global shocks to the entire system, there are some emerging studies of MFG and MFC in the presence Poissonian common noise. For instance, Hern\'andez-Hern\'andez and Ricalde-Guerrero \cite{HHRG2023,HHRG2024} investigated the propagation of chaos and stochastic maximum principle for MFG with Poissonian common noise. Bo et al. \cite{BWWY} studied the stochastic maximum principle and the HJB equation under open-loop controls for extended MFC with Poissonian common noise. However, it remains an interesting open problem that whether these problems with Poissonian common noise can be addressed by compactification arguments. In response, the present paper aims to propose new techniques in employing the compactification approach without the discretization procedure but by taking advantage of the point process representation of Poisson random measure with finite intensity. Our main result stands out in the literature using the compactification approach as the desired adaptivity with respect to common noise filtration can be retained.
 
 More precisely, in MFC under the assumption that Poissonian common noise has finite intensity, we introduce an auxiliary probabilistic setup by fixing an arbitrary sample path in the canonical space  $\omega^1\in\Omega^1$  to support the common noise. This is possible thanks to the assumption of finite intensity of the Poisson random measure such that the pathwise construction of the stochastic integral with respect to the Poisson measure is well defined and each sample path only exhibits finitely many jumps over the finite time horizon; see \remref{MFC_extension} for more details. By doing so, we can exercise our pathwise formulation approach in two main steps. In Step-1, we first consider the pathwise MFC formulation without common noise as an auxiliary martingale problem with associated admissible pathwise relaxed controls (see \defref{relaxed_control_pathwise} and the problem \eqref{eq:RMopt}) when the jump terms become deterministic jumps. The rationale behind the pathwise formulation is the conjectured equivalence in \eqref{value_equivalent} between the original relaxed control problem with Poissonian common noise and the aggregation of pathwise relaxed control problems over all sample paths. In this step, we can perform compactification (\propref{existence_pathwise}) arguments in the auxiliary model in the Skorokhod topology as if common noise is absent but
 with deterministic jumping times, which produces an optimal control $P_*^{\omega^1}$ as a measurable mapping from $\Omega^1$ to the optimal pathwise relaxed control set. 
 
In Step-2, the task is to verify the key conjecture of equivalence in \eqref{value_equivalent}. To this end, we utilize the Fokker-Planck equation to heuristically transform the strict control problem with a fixed sample path $\omega^1\in\Omega^1$ into a pathwise measure-valued control problem. By means of concatenation techniques over a sequence of deterministic jumping times, we establish a pathwise superposition principle (\thmref{thm:equivalence}-{\bf (ii)}), confirming the relationship between the pathwise measure-valued control problem and the pathwise relaxed control problem when the sample path of common noise is fixed. Based on some standard approximation arguments, we can obtain the equivalence between the strict control problem and the relaxed control problem in the original model with Poissonian common noise  (\thmref{thm:equivalence}-{\bf (i)}). We finally prove the desired equivalence  \eqref{value_equivalent} in \thmref{thm:equivalence}-{\bf (iii)} via two sided inequalities: On one hand, Lemma 6.14 in \cite{BWY} implies that the value function of the original problem with common noise is less than that of the pathwise formulation; on the other hand, the reverse inequality follows by considering the admissible control $\bar{P}^*(\mathrm{d}\omega,\mathrm{d}\omega^1) = P^{\omega^1}_*(\mathrm{d}\omega)P^1(\mathrm{d}\omega^1)$ together with the established superposition principle in the pathwise formulation in \thmref{thm:equivalence}-{\bf (ii)}.  Consequently, the equivalence in \eqref{value_equivalent} can be concluded such that $\bar{P}^*(\mathrm{d}\omega,\mathrm{d}\omega^1) = P^{\omega^1}_*(\mathrm{d}\omega)P^1(\mathrm{d}\omega^1)$ constitutes an optimal relaxed control in the original problem (\thmref{thm:equivalence}-{\bf (iii)}).

\begin{comment}
Our pathwise compactification approach is also directly applicable in solving MFG problems with Poissonian common noise. Similar to the case of MFC, we can again freeze the sample path of the Poissonian common noise and consider the pathwise relaxed control problem for MFG in the auxiliary setup with deterministic jumping times; see \defref{pathwise_MFG}. Using the standard compactification arguments in the pathwise formulation without common noise, the existence of pathwise MFEs (see \defref{pathwise-mfe}) is guaranteed. Then, by aggregation over all sample paths, we show in \thmref{existence_MFG} that the pair $(\bar{\boldsymbol{\mu}}^*, \bar{P}^*)$ constitutes a strong MFE, where the probability measure $\bar{P}^*$ on $\Omega \times \Omega^1$ is constructed by $\bar{P}^*(\mathrm{d}\omega, \mathrm{d}\omega^1) := P_*^{\omega^1}(\mathrm{d}\omega) P^1(\mathrm{d}\omega^1)$ and the c\`{a}dl\`{a}g $\Fb^1$-adapted measure flow $\bar{\boldsymbol{\mu}}^*=(\bar\mu^*_t)_{t\in [0,T]}$ is constructed by $\bar{\mu}^*_t(\omega^1) := \mu^{\omega^1}_t$ for all $(t,\omega^1) \in [0,T] \times \Omega^1$. We again highlight that the obtained MFE using the pathwise formulation approach is of the strong type, i.e., the MFE is indeed common noise adapted. 
\end{comment}

The rest of the paper is organized as follows. Section \ref{sec:prob} introduces the model with Poissonian common noise and the relaxed control problem formulation of the MFC. Section \ref{sec:Compactification} establishes the existence of pathwise optimal controls using the compactification arguments in the pathwise formulation as if the common noise is absent. Section \ref{sec:Equivalence} develops the equivalence between the original problem with Poissonian common noise and the pathwise formulation with the aid of the auxiliary measure valued control problem, thereby confirming the existence of optimal relaxed controls in the original model.  Section \ref{auxiliary} collects some auxiliary results and proofs.

\vspace{0.1in}
\noindent{\bf Notations.}\quad We list below some notations that will be used frequently throughout the paper:
\vspace{-0.35in}
\begin{center}
\begin{longtable}{l l}
$|\cdot|$ & Euclidean norm on $\R^n$\\
%$\C$ & The continuous function space $C([0,T];\R)$\\
%	$C([0,T];E)$ & The set of continuous functions $f:[0,T]\to E$\\
%$\lv\cdot\rv_{\infty}$ & The supremum norm on $\C$\\
%	$f'(f'')$ & The first (second) derivative of a function $f:\R\to\R$.\\
$L^p((A,\mathscr{B}(A),\lambda_A);E)$ & Set of $L^p$-integrable $E$-valued mapping defined on $(A,\mathscr{B}(A))$\\ 
& we write $L^p(A;E)$ for short\\
$\nabla_i \phi$ & Partial derivative of $\phi$ w.r.t. the $i$-th component of argument\\ 
$\Law^P(\kappa)$ ($\E^P[\kappa]$) & Law (Expectation) of r.v. $\kappa$ under probability measure $P$\\
$\Pc_p(E)$ & Set of probability measures on $E$ with finite $p$-order moments\\
$M_p(\mu)$ & $\left(\int_{\R^n}|e|^p\mu(\d e)\right)^{\frac1p}$ for $\mu\in\Pc_p(\R^n)$\\  
$\mathcal{W}_{p,E}$ & The $p$-Wasserstein metric on $\Pc_p(E)$\\
$\Mc(E)$ & Set of signed Randon measures on $E$\\
$\Mc_c(E)$ & Set of simple finite counting measures on $E$\\
$C([0,T];E)$ & Set of $E$-valued continuous functions on $[0,T]$\\
$D([0,T]; E)$ & Set of $E$-valued c\`{a}dl\`{a}g functions on $[0,T]$\\
%$C_b^2(\R^n)$ & Set of continuous and bounded functions $\phi: \R^n \to \R$ such that \\
%                & $\nabla_x \phi$ and $\nabla_{x}^2\phi$ exist, and are continuous and bounded\\
$\langle \phi, \mu\rangle$ & $\int_{\R^n} \phi(x) \mu(dx)$ for $\mu \in \Pc_2(\R^n)$ and integrable function $\phi:\R^n \to \R$\\

{$\mathsf{R}$ ($\mathsf{R}^{\rm s}$)} & {Set of admissible relaxed (strict) controls (see Def. \ref{relaxed_control})}\\
{$\mathsf{R}(\omega^1)$ ($\mathsf{R}^{\rm s}(\omega^1)$)} & {Set of pathwise admissible relaxed (strict) controls (see Def. \ref{relaxed_control_pathwise})}\\
{$\mathsf{R}^{\rm opt}$ ($\mathsf{R}^{\rm opt}(\omega^1)$)} & {Set of optimal (pathwise) admissible relaxed controls}

\end{longtable}
\end{center}

\section{Problem Formulation}\label{sec:prob}
We first introduce a standard strict control formulation in the strong sense. Let $T>0$ be a finite horizon and $(\Omega,\F,\Fb,\Pb)$ be a filtered probability space with the filtration $\Fb=(\F_t)_{t\in[0,T]}$ satisfying the usual conditions. For $n,l\in\mathbb{N}$ and $p>2$, let $W=(W_t)_{t\in[0,T]}$ be a standard $n$-dimensional ($\Pb,\Fb$)-Brownian motion and $N(\d t,\d z)$ be a $(\Pb,\Fb)$-Poisson random measure on a measurable space {$([0,T]\times Z,{\cal B}([0,T])\otimes\mathscr{Z})$, whose intensity measure is given by $\nu(\d z)dt$ with $(Z,\mathscr{Z})$ being a measurable space satisfying $\nu(Z)<\infty$}. The control space $U\subset\R^l$ is assumed to be compact  and $\U[0,T]$ denotes the set of admissible controls which are $\Fb$-progressively measurable processes. We also set $\Fb^N=(\F_t^N)_{t\in [0,T]}$ with $\F_t^N:=\sigma(N((0,s]\times A);~s\leq t,A\in\mathscr{Z})$. Assume that coefficients $(b,\sigma,f):[0,T]\times\R^n\times\Pc_2(\R^n)\times U\to\R^n\times\R^{n\times n}\times\R$ and $\gamma:[0,T]\times\R^n\times\Pc_2(\R^n)\times Z\to\R^n$ are Borel measurable. The initial data $\kappa\in L^p((\Omega,\F_0,\Pb),\R^n)$ is independent of $(W,N)$ with law $\lambda\in\Pc_p(\R^n)$, i.e., $\lambda={\cal L}^{\mathbb{P}}(\kappa)$. For an admissible control $\alpha=(\alpha_t)_{t\in [0,T]}\in\U[0,T]$, let us consider the controlled conditional McKean-Vlasov dynamics:
\begin{align}\label{double_reflection}
\d X_t^{\alpha}=b(t,X_t^{\alpha},\mu_t,\alpha_t)\d t+\sigma(t,X_t^{\alpha},\mu_t,\alpha_t)\d W_t+\int_Z\gamma(t,X_{t-}^{\alpha},\mu_{t-},z)N(\d t,\d z),~ X_0^{\alpha}=\kappa,
\end{align}
where $\mu_t:=\Law^{\Pb}(X_t^{\alpha}|\F_t^N)$ is the conditional distribution of $X_t^{\alpha}$ at time $t\in(0,T]$ and the Poisson random measure plays the role of common noise. 
    
Due to the fact that $N$ is a $(\Pb,\Fb)$-Poisson random measure, one can easily verify that, for any $t\in [0,T]$,
\begin{align}\label{compatible}
\E^{\Pb}\left[\boldsymbol{1}_D|\F_t^N\right]=\E^{\Pb}\left[\boldsymbol{1}_D|\F_T^N\right],~\Pb\text{-}\as,~~\forall  D\in\F_t\vee\F_T^W
\end{align}
with $\F_T^W:=\sigma(W_t;~0\leq t\leq T)$. In particular, it holds that $\Law^{\Pb}(X_t^{\alpha}|\F_t^N)=\Law^{\Pb}(X_t^{\alpha}|\F_T^N)$ for $t\in [0,T]$. The equality \equref{compatible} is often referred as the compatibility condition in the mean field theory with common noise (c.f. Eq.~(2.5) in Djete et al. \cite{Djete} for MFC, and Definition 1.6 in Carmona and Delarue \cite{Carmona} for MFG).
	
	The goal of the social planner in the MFC problem is to minimize the following cost functional over $\alpha\in\U[0,T]$,
	\begin{align}\label{cost_func}
		J(\alpha):=\E^{\Pb}\left[\int_0^Tf(t,X_t^{\alpha},\mu_t,\alpha_t)\d t{+g(X_T^{\alpha},\mu_T)}\right].
	\end{align}
	
	\begin{definition}\label{optimal_control}
		We call	$\alpha^*\in\U[0,T]$ an  optimal (strict) control (in the strong sense) if it holds that $J(\alpha^*)=\inf_{\alpha\in\U[0,T]}J(\alpha)$.
		%\begin{align}\label{optimality}
		%			J(\alpha^*)=\inf_{\alpha\in\U[0,T]}J(\alpha).
		%		\end{align} 
\end{definition}
We impose the following assumptions on model coefficients throughout the paper.
\begin{ass}\label{ass1} 
\begin{itemize}
\item[{\rm(A1)}]  The coefficients $(b,\sigma,f):[0,T]\times\R^n\times\Pc_2(\R^n)\times U\to \R^n\times\R^{n\times n}\times\R$, {$g:\R^n\times\Pc_2(\R^n)\to\R$} and $\gamma:[0,T]\times\R^n\times\Pc_2(\R^n)\times Z\to\R^n$ are jointly continuous and $(b,\sigma,f)$ are all uniformly continuous in $u\in U$ with respect to $(t,x,\mu)\in[0,T]\times\R^n\times\Pc_2(\R^n)$.
		
\item[{\rm(A2)}] The coefficients $(b,\sigma,\gamma)$ are uniformly Lipschitz continuous in $(x,\mu)\in\R^n\times\Pc_2(\R^n)$ in the sense that, there exists a constant  $M>0$ independent of $(t,u,z)\in [0,T]\times U\times Z$ such that, for all $(x,\mu),(x',\mu')\in\R^n\times\Pc_2(\R^n)$,
\begin{align*}
&\left|(b,\sigma)(t,x',\mu',u)-(b,\sigma)(t,x,\mu,u)\right|+\left|\gamma(t,x',\mu',z)-\gamma(t,x,{\color{red}\mu},z)\right|\nonumber\\
&\qquad\leq M(|x-x'|+\mathcal{W}_{2,\R^n}(\mu,\mu')).
\end{align*}
		%where $\mathcal{W}_{2,\R^n}$ is the $2$-Wasserstein metric on $\Pc_2(\R^n)$.
\item[\rm(A3)] There exists a constant $M>0$ independent of $(t,u)\in[0,T]\times U$ such that, for all $(x,\mu),(x',\mu')\in\R^n\times\Pc_2(\R^n)$, 
\begin{align*}
{|g(x',\mu')-g(x,\mu)|}+|f(t,x',\mu',u)-f(t,x,\mu,u)|\leq M\left(1+|x-x'|^2+\mathcal{W}_{2,\R^n}(\mu,\mu')^2\right).
\end{align*}

\item [{\rm (A4)}] There exists a constant $M>0$ independent of $(t,z)\in [0,T]\times Z$ such that $|\gamma(t,x,\mu,z)|\leq M(1+|x|+M_2(\mu))$ for all $(x,\mu)\in\R^n\times\Pc_2(\R^n)$.
	\end{itemize}
\end{ass}

As a preparation for different problem formulations, let us also introduce some basic spaces: 
\begin{itemize}

\item The space {$\D:=D([0,T];\R^n)$} is endowed with the Skorokhod metric $d_{\D}$\ and the Borel $\sigma$-algebra is denoted by $\F^X$, and $\F_t^X$ stands for the Borel $\sigma$-algebra up to time $t$. 

\item 	The space $\mathcal{Q}$ of relaxed controls is defined as the set of measures $q$ in $[0,T]\times U$ with the first marginal equal to the Lebesgue measure and $\int_{[0,T]\times U}|u|^pq(\d t,\d u)<\infty$. We endow the space $\mathcal{Q}$ with the $2$-Wasserstein metric on $\Pc_2([0,T]\times U)$ given by $d_{\mathcal{Q}}(q^1,q^2)=\mathcal{W}_{2,[0,T]\times U}(\frac{q^1}{T},\frac{q^2}{T})$, where the metric on $[0,T]\times U$ is given by $((t_1,u_1),(t_2,u_2))\mapsto|t_2-t_1|+|u_2-u_1|$. Note that, each $q\in\mathcal{Q}$ can be identified with a measurable function $[0,T]\in t\mapsto q_t\in\Pc_2(U)$, defined uniquely up to $\as$ by $q(\d t,\d u)=q_t(\d u)\d t$. In the sequel, we will always refer to the measurable mapping $q=(q_t)_{t\in [0,T]}$ to a relaxed control in $\mathcal{Q}$. Let $\F^{\mathcal{Q}}$ be the Borel $\sigma$-algebra of $\mathcal{Q}$ and $\F_t^{\mathcal{Q}}$ be the $\sigma$-algebra generated by the maps $q\mapsto q([0,s]\times V)$ with $s\in [0,t]$ and Borel measurable $V\subset U$. Because $U$ is compact and Polish, $\mathcal{Q}$ as a closed subset of $\Pc_2([0,T]\times U)$ is also compact and Polish.

\item The space {$\C:=C([0,T];\R^n)$} is endowed with the supremum norm $\|\cdot\|_{\infty}$ and the Borel $\sigma$-algebra is denoted by $\F^W$, and $\F_t^W$ stands for the Borel $\sigma$-algebra up to time $t$.
    
\item Denote by $\Pi_{Z}$ the collection of point functions $p:D_{p}\subset [0,T]\to Z$ with $D_{p}$ being a finite set (see Section 1.9 in \cite{Ikeda} for a detailed definition of point functions). As stated therein, each point function $p \in \Pi_Z$ induces a counting measure $N_p(\d t, \d z)$ on $[0,T] \times Z$ via the injective mapping $\mathscr{N}:\Pi_Z\to \Mc_c([0,T]\times Z)$, $p\to N_p(\d t,\d z)$, where $N_p([0,t]\times A)=\#\{s\in D_p;~s\leq t,~p(s)\in A\}$ for $t\in[0,T]$ and $A\in\mathscr{Z}$.
	
\item The space $\Omega^1:=\mathscr{N}(\Pi_Z)$, i.e., the image of $\Pi_Z$ under the injective mapping $\mathscr{N}$. {It is identified as a subset of $L^{\infty}([0,T]\times Z;\R)$ and is endowed with the corresponding weak* topology, $\ie$, $\omega^{1,n}\to\omega^1$ if and only if for any $\phi\in L^{\infty}([0,T]\times Z)$, 
\begin{align*}
\int_0^T\int_Z \phi(t,z)\omega^{1,n}(\d t,\d z)\to \int_0^T\int_Z\phi(t,z)\omega^1(\d t,\d z),~~n\to\infty.
\end{align*}}Denote by $\F^1$ the Borel $\sigma$-algebra on $\Omega^1$. For any $\omega^1\in\Omega^1$, we set $p^{\omega^1}=\mathscr{N}^{-1}(\omega^1)$. Define the filtration $\Fb^1 = (\F_t^1)_{t \in [0,T]}$ by ${\F_t^1}:=\sigma(N((0,t]\times A);~t\in [0,T],A\in\mathscr{Z})$ for $t\in[0,T]$, and $N(\omega^1)=\omega^1$ for all $\omega^1\in\Omega^1$, i.e., the identity mapping on $\Omega^1$. Moreover, let $P^1$ be the probability measure on $(\Omega^1,\F^1)$ under which $N$ is an $\Fb^1$-Poisson random measure with intensity $\nu(\d z)dt$.  
\end{itemize}

Define the canonical spaces $\Omega=\D\times\mathcal{Q}\times\C$ and $\bar{\Omega}=\Omega\times\Omega^1$. Endow them with the respective (product) $\sigma$-algebra $\F=\F^{X}\otimes\F^{\mathcal{Q}}\otimes\F^W$ and $\bar{\F}=\F\otimes\F^1$. The corresponding product filtrations are given by $\F_t=\F_t^X\otimes\F_t^{\mathcal{Q}}\otimes\F_t^W$ and $\bar{\F}_t=\F_t\otimes\F_t^1$ for $t\in[0,T]$.  In particular, $\Omega$ is Polish under the metric defined by $d_{\Omega}(\omega_1,\omega_2):=d_{\D}(\boldsymbol{x}_1,\boldsymbol{x}_2)+d_{\mathcal{Q}}(q^1,q^2)+\|\boldsymbol{w}_1-\boldsymbol{w}_2\|_{\infty}$ for $\omega_i=(\boldsymbol{x}_i,q^i,\boldsymbol{w}_i)\in\Omega$ with $i=1,2$. Moreover, we also introduce the coordinate mappings $(X,\Lambda,W)=(X_t,\Lambda_t,W_t)_{t\in [0,T]}$ and $(\bar X,\bar\Lambda,\bar W,\bar N)=((\bar{X}_t)_{t\in [0,T]},(\bar{\Lambda}_t)_{t\in [0,T]},(\bar W_t)_{t\in [0,T]},\bar{N}(\d t,\d z))$ as, for $\omega=(\boldsymbol{x},q,\boldsymbol{w})\in\Omega$ and $\bar{\omega}=(\boldsymbol{x},q,\boldsymbol{w},\omega^1)\in\bar\Omega$, 
\begin{align}\label{eq:coordinatemappings}
\bar{X}_t(\bar{\omega})= X_t(\omega)=\boldsymbol{x}(t),~&\bar{\Lambda}_t(\bar{\omega})=\Lambda_t(\omega)=q_t,~\bar{W}_t(\bar\omega)=W_t(\omega)=\boldsymbol{w}(t),\nonumber\\
&\bar{N}(\bar{\omega})(\d t,\d z)=\omega^1(\d t,\d z).   
\end{align}
For simplicity, denote by $\F_t^X$, $\F_t^{\mathcal{Q}}$, $\F_t^W$ and $\F_t^1$ for $t\in[0,T]$ the natural extensions of these filtrations to $\Omega$ and $\bar{\Omega}$. In the sequel, when talking about the filtrations $\F_t^X$, $\F_t^{\mathcal{Q}}$, $\F_t^W$ and  $\F_t^1$ for $t\in[0,T]$, there should be no confusion of which space the filtrations are defined on.

We next give the definition of admissible relaxed control rules in the model with Poissonian common noise.
\begin{definition}[Relaxed Control in Original Problem]\label{relaxed_control}
We call a probability measure $\bar{P}\in\Pc_2(\bar\Omega)$ on $(\bar\Omega,\bar\F)$ an  admissible relaxed control rule (denoted by $\bar P\in \mathsf{R}$) if it holds that {\rm(i)}  $\bar P\circ \bar{X}_0^{-1}=\lambda$, $\bar P(\bar W_0=0)=1$ and $\bar{X}_0$ is independent of $(\bar{W},\bar{N})$ under $\bar{P}$;  {\rm(ii)} the restriction of $\bar P$ to $\Omega^1$ $\bar P|_{\Omega^1}$ agrees with the law of $N$ under $\Pb$ on $(\Omega^1,\F^1)$, i.e., $\bar P|_{\Omega^1}=\Pb\circ \bar{N}^{-1}:=P^1$; {\rm (iii)} there exists an $\mathbb{F}^1$-adapted c\`{a}dl\`{a}g $\Pc_2(\R^n)$-valued process $\bar{\boldsymbol{\mu}}=(\bar\mu_t)_{t\in [0,T]}$ such that $\bar{P}(\bar{\mu}_t=\Law^{\bar P}(\bar{X}_t|\F_t^1),\forall t\in [0,T])=1$;
{\rm(iv)} {for any test function $\phi\in C^2_b(\R^n\times\R^n)$}, the process
\begin{align*}
{\tt M}^{\bar P}\phi(t):&=\phi(\bar{X}_t,\bar{W}_t)-\int_0^t\int_U\bar{\mathbb{L}}\phi(s,\bar{X}_s,\bar{W}_s,\bar\mu_s,u)\bar\Lambda_s(\d u)\d s\\
&\quad-\int_0^t\int_Z\left(\phi(\bar{X}_{s-}+\gamma(s,\bar{X}_{s-},\bar\mu_{s-},z),\bar W_s)-\phi(\bar{X}_{s-},\bar{W}_s)\right)\bar{N}(\d s,\d z),\quad t\in[0,T]
\end{align*}
is a $(\bar P,\bar \Fb)$-martingale. Here, the infinitesimal generator acting on $\phi\in C_b^2(\R^n\times\R^n)$ is defined by, for $(t,x,\mu,u)\in[0,T]\times\R^n\times\Pc_2(\R^n)\times\R^l$, 
\begin{align*}
\bar{\mathbb{L}}\phi(t,x,w,\mu,u):=\bar b(t,x,\mu,u)^{\T}\nabla\phi(x,w)+\frac12\tr\left(\bar\sigma\bar\sigma^{\T}(t,x,\mu,u)\nabla^2\phi(x,w)\right),
\end{align*}
where the coefficients
\begin{align*}
\bar b(t,x,\mu,u)=\begin{pmatrix}
b(t,x,\mu,u)\\
\boldsymbol{0}_n
\end{pmatrix},\quad\bar\sigma(t,x,\mu,u)=\begin{pmatrix}
\sigma(t,x,\mu,u)\\
I_{n}
\end{pmatrix}
\end{align*}
with $\boldsymbol{0}_n$ and $I_n$ being the zero vector in $\R^n$ and the identity matrix in $\R^{n\times n}$, respectively. Furthermore, if there exists an $\bar{\mathbb{F}}$-progressively measurable $U$-valued process $\bar\alpha=(\bar\alpha_t)_{t\in [0,T]}$ on $\bar\Omega$ such that $\bar P(\bar\Lambda_t(\d u)\d t=\delta_{\bar\alpha_t}(\d u)\d t)=1$, we say that $\bar P$ corresponds to a strict control $\bar\alpha$ or we call it a strict control rule. The set of all strict control rules is denoted by $\mathsf{R}^{\rm s}$.
\end{definition}
We have the following martingale measure characterization and moment estimate for admissible relaxed controls, whose proof is standard and omitted.
\begin{lemma}\label{moment_p}
$\bar{P}\in\mathsf{R}$ iff there exists a filtered probability space $(\Omega',\F',\Fb'=(\F'_t)_{t\in[0,T]},P')$ supporting a $\Pc(U)$-valued $\Fb'$-progressively measurable process $\bar\Lambda=(\bar\Lambda_t)_{t\in [0,T]}$, an $\R^n$-valued $\Fb'$-adapted process $\bar{X}^{\bar\Lambda}=(\bar{X}_t^{\bar\Lambda})_{t\in[0,T]}$, an $n$-dimensional standard $(P',\Fb')$-Brownian motion $\bar W=(\bar W_t)_{t\in [0,T]}$, an $\R^n$-valued $\Fb'$-martingale measure $\bar{{\cal M}}$ on $[0,T]\times U$, with the intensity $\bar\Lambda_t(\d u)\d t$ and a Poisson random measure $\bar{N}(\d t,\d z)$ satisfying $P'\circ \bar{N}^{-1}=P^1$  such that $\bar{P}=P'\circ (\bar{X}^{\Lambda},\bar{\Lambda},\bar{W},\bar{N})^{-1}$,  and it holds that {\rm(i)} $P'\circ(\bar{X}_0^{\Lambda})^{-1}=\lambda$; {\rm(ii)} $\bar{X}_0^{\Lambda},\bar{W}$ and $\bar{N}$ are independent under $P'$, and it holds that $P'$-$\as$, $\bar W_t=\int_0^t\int_U\bar{\cal M}(\d u,\d s)$; {\rm(iii)} the dynamics of state process $\bar{X}^{\bar\Lambda}$ obeys that, $P'$-a.s.,
\begin{align*}
\d \bar{X}^{\bar\Lambda}_t=\int_Ub(t,\bar{X}^{\bar\Lambda}_t,\mu_t,u)\bar\Lambda_t(\d u)\d t+\int_U\sigma(t,\bar{X}^{\bar\Lambda}_t,\mu_t,u)\bar{{\cal M}}(\d u,\d t)+\int_Z\gamma(t,\bar{X}^{\bar\Lambda}_{t-},\mu_{t-},z) \bar{N}(\d t,\d z).
\end{align*}         
Here, for $t\in[0,T]$, $\mu_t:=\Law^{P'}(\bar{X}^{\Lambda}_t|\F_t^{\bar{N}})$ where $(\F_t^{\bar{N}})_{t\in[0,T]}$ denotes the augmentation filtration of the natural filtration $\sigma(\bar{N}((0,s]\times A);~s\leq t , A\in\mathscr{Z})$ so that $(\F_t^{\bar N})_{t\in[0,T]}$ satisfies the usual conditions. Moreover, there exists a constant $C>0$ depending on $M,M_p(\lambda)$ and $T$ such that 
\begin{align}\label{eq:moment_p}
\E^{P'}\left[\sup_{t\in[0,T]}\left|\bar{X}^{\Lambda}_t\right|^p\right]\leq C
\end{align}
with $M$ being stated in \assref{ass1}.
\end{lemma}

Consider the coordinate mappings defined in \eqref{eq:coordinatemappings}. The cost functional of our MFC problem is defined by\begin{align}\label{value_common}
	\mathcal{J}(\bar P) &:=\E^{\bar P}\left[\int_0^T\int_Uf(t,\bar{X}_t,\bar\mu_t,u)\bar\Lambda_t(\d u)\d t{+g(\bar{X}_T,\bar\mu_T)}\right],\quad \forall\bar P\in \mathsf{R},
\end{align}
where $\bar{\boldsymbol{\mu}}=(\bar\mu_t)_{t\in [0,T]}$ is the associated $\Fb^1$-adapted measure flow associated to $\Pb$ (see \defref{relaxed_control}). Denote by ${\mathsf{R}^{\rm opt}}:=\mathop{\arg\min}_{\bar{P}\in \mathsf{R}}\mathcal{J}(\bar{P})$ the set of optimal {relaxed} control rules.

\begin{remark}\label{relaxed_control_extension}
Note that, for any $\bar P\in \mathsf{R}$, the push forward measure $\bar P\circ \left(\bar{X},\bar\Lambda,\bar W,\bar N,\bar{\boldsymbol{\mu}}\right)^{-1}$ induces a probability measure on $\bar{\Omega}\times D([0,T];\Pc_2(\R^n))$. In view of this fact, we can give an equivalent formulation of \defref{relaxed_control}. We first extend $\bar\Omega$ to $\hat\Omega:=\bar\Omega\times D([0,T];\Pc_2(\R^n))$ and equip it with the product metric $d_{\hat\Omega}(\hat\omega^1,\hat\omega^2)=d_{\bar\Omega}(\bar\omega^1,\bar\omega^2)+d_{D([0,T];\Pc_2(\R^n))}(\boldsymbol{\mu}^1,\boldsymbol{\mu}^2)$ for $\hat\omega^i=(\bar\omega^i,\boldsymbol{\mu}^i)\in\hat\Omega,i=1,2$. Denote by $\hat\F$ the Borel $\sigma$-algebra on $\hat\Omega$ (also the product $\sigma$-algebra). Furthermore, we define the filtration $\Fb^{\boldsymbol{\mu}}=(\F_t^{\boldsymbol{\mu}})_{t\in [0,T]}$ by $\F_t^{\boldsymbol{\mu}}=\sigma\left(\mu_s(A),s\leq t,A\in\mathcal{B}(\R^n)\right)$ and then define the product filtration $\hat\Fb=(\hat\F_t)_{t\in [0,T]}$ with $\hat\F_t=\bar\F_t\otimes\F_t^{\boldsymbol{\mu}}$. Denote by $(\hat{X},\hat{\Lambda},{\hat{W}},\hat{N},\hat{\boldsymbol{\mu}})$ the corresponding coordinate mapping, $\ie$, for $\hat{\omega}=(\boldsymbol{x},q,\boldsymbol{w},\omega^1,\boldsymbol{\mu})\in\hat\Omega$,
\begin{align*}
\hat{X}_t(\hat\omega)=\boldsymbol{x}(t),~\hat\Lambda(\hat\omega)=q_t,~\hat W_t(\hat\omega)=\boldsymbol{w}(t),~\hat{N}(\hat\omega)=\omega^1(\d t,\d z),~\hat{\mu}_t=\mu_t~\text{for $t\in[0,T]$}.
\end{align*}
We still denote by $\Fb^X,\Fb^{\mathcal{Q}},\Fb^W,\Fb^1,\Fb^{\boldsymbol{\mu}}$ the the natural extensions of these filtrations to $\hat\Omega$ for simplicity. 
Then, one can easily verify that $\bar P\in \mathsf{R}$ iff there exists a probability measure $\hat P\in\Pc_2(\hat\Omega)$ such that {\rm(i)} $\hat P\circ \hat{X}_0^{-1}=\lambda$, $\hat P(\hat W_0=0)=1$ and $\hat{X}_0$ is independent of $(\hat{W},\hat{N})$ under $\hat{P}$; {\rm(ii)} the restriction of $\hat P$ to $\Omega^1$ satisfies $\hat P|_{\Omega^1}=P^1$; {\rm (iii)} $\hat{P}(\hat\mu_t=\Law^{\hat P}(\hat{X}_t|\F_t^1),\forall t\in [0,T])=1$; {\rm(iv)} for any test function $\phi\in C_b^2(\R^n\times\R^n)$, the process\begin{align*}
\hat{\tt M}^{\hat P}\phi(t):&=\phi(\hat{X}_t,\hat{W}_t)-\int_0^t\int_U\bar{\mathbb{L}}\phi(s,\hat{X}_s,\hat{W}_s,\hat\mu_s,u)\hat\Lambda_s(\d u)\d s\\
&\quad-\int_0^t\int_Z\left(\phi(\hat{X}_{s-}+\gamma(s,\hat{X}_{s-},\hat\mu_{s-},z),\hat{W}_s)-\phi(\hat{X}_{s-},{\hat{W}_s})\right)\hat{N}(\d s,\d z),\quad t\in[0,T]
\end{align*}
	is a $(\hat P,\hat\Fb)$-martingale; {\rm(v)} $\bar{P}=\hat{P}\circ (\bar{X},\bar{\Lambda},\bar{W},\bar{N})^{-1}$. Such subset of $\Pc_2(\hat\Omega)$ is denoted by {$\hat{\mathsf{R}}$}. The corresponding cost functional is defined by\begin{align}\label{hat_cost}
		\hat{J}(\hat{P}):=\E^{\hat P}\left[\int_0^T\int_Uf(t,\hat{X}_t,\hat{\mu}_t,u)\hat{\Lambda}_t(\d u)\d t{+g(\hat X_T,\hat\mu_T)} \right],\quad \forall\hat{P}\in{\hat{\mathsf{R}}}.
	\end{align}
Moreover, if there exists an $\hat{\mathbb{F}}$-progressively measurable $U$-valued process $\hat\alpha=(\hat\alpha_t)_{t\in [0,T]}$ on $\hat\Omega$ such that $\hat P(\hat\Lambda_t(\d u)\d t=\delta_{\hat\alpha_t}(\d u)\d t)=1$, we say that $\hat P$ corresponds to a strict control $\hat\alpha$ or it is called a strict control rule. The set of all strict control rules is denoted by $\hat{\mathsf{R}}^{\rm s}$.
    
\end{remark}
The next theorem is the main result for the MFC problems.
\begin{theorem}\label{existence}
	Let \assref{ass1} hold. The optimal {relaxed} control set $\mathsf{R}^{\rm opt}$ is nonempty.
\end{theorem}
Its proof consists of two main steps using our pathwise compactification approach, which are detailed later in Section \ref{sec:Compactification} and Section \ref{sec:Equivalence}. In a nutshell, 
\begin{itemize}
\item[(i)] In Step-1, we first consider an auxiliary model, called the pathwise formulation, by freezing a sample path of common noise. In this step, we can modify the classical compactification arguments in the Skorokhod topology in the model without common noise but with finite deterministic jumping times and obtain the existence of an optimal pathwise relaxed control. We further verify the measurability of the optimal solution with respect to the sample path to facilitate the aggregation form over all sample paths. 

\item[(ii)] In Step-2, we address the main challenge in our pathwise formulation approach, that is, whether the aggregation of the optimal pathwise relaxed controls over all sample paths of common noise is an optimal solution in the original model. To achieve this goal, we introduce the pathwise measure-valued control problem and establish a pathwise superposition principle in the auxiliary model with deterministic jumping times to bridge the desired equivalence between the pathwise formulation and the original problem.   
\end{itemize}

Moreover, we also have existence of {an optimal strict} control under the additional convexity assumption.
\begin{ass}\label{ass2}
	For any $(t,x,\mu)\in[0,T]\times\R^n\times\Pc_2(\R^n)$, the following set is convex in $\R^n\times\R^{n\times n}\times\R$:
	\begin{align*}
		K(t,x,\mu):=\left\{(b(t,x,\mu,u),\sigma\sigma^{\T}(t,x,\mu,u),z);~z\geq f(t,x,\mu,u),~u\in U\right\}.
	\end{align*}
\end{ass}
Then, we have the next corollary whose proof is standard (c.f. Corollary 3.8 in \cite{Lacker3}).
\begin{corollary}
	Let  \assref{ass1} and \assref{ass2} hold. There exists a strict control {$\bar{P}^{\rm s}\in \mathsf{R}^{\rm s}$}. 
\end{corollary}

%{\color{red}Above, $\mathsf{R}^{\rm s}$ is strict control set; while $R^{\rm opt}(\lambda)$ is the relaxed control set, how to make $ \mathsf{R}^{\rm s}\cap R^{\rm opt}(\lambda)$ sense?}

\section{Step-1: Compactification in Pathwise Formulation}\label{sec:Compactification}

This section presents the first step of the proof for \thmref{existence}, for which we leverage the probabilistic characteristics of the Poisson random measure and introduce a novel pathwise formulation as if there is no common noise. We then establish the existence of optimal solutions in the pathwise formulation. 

\subsection{Pathwise formulation}\label{pathwise_formulation}
We first introduce the pathwise problem formulation and the corresponding pathwise admissible relaxed control rules by fixing an arbitrary sample path $\omega^1\in\Omega^1$.

\begin{definition}[Pathwise Relaxed Control (without common noise)]\label{relaxed_control_pathwise}
Let $\omega^1\in\Omega^1$ be fixed. We call a probability measure $P^{\omega^1}\in\Pc_2(\Omega)$ on $(\Omega,\F)$ a pathwise  admissible relaxed control rule (denoted by $P^{\omega^1}\in \mathsf{R}(\omega^1)$) if it holds that {\rm(i)}  $P^{\omega^1}(W_0=0)=1$, $ P^{\omega^1}\circ X_0^{-1}=\lambda$ and $X_0$ is $P^{\omega^1}$-independent of $W$;  {\rm(ii)} for any test function $\phi\in C^2_b(\R^n\times\R^n)$, the process
\begin{align*}
{\tt M}^{\omega^1,P^{\omega^1}}\phi(t):&=\phi(X_t,W_t)-\int_0^t\int_U\bar{\mathbb{L}}\phi(s,X_s,W_s,\mu_s^{\omega^1},u)\Lambda_s(\d u)\d s\\
&\quad-\int_0^t\int_Z\left(\phi(X_{s-}+\gamma(s,X_{s-},\mu_{s-}^{\omega^1},z),W_s)-\phi(X_{s-},W_s)\right)\omega^1(\d s,\d z),~~ t\in[0,T]
\end{align*}
is a $(P^{\omega^1},\Fb)$-martingale, where $\mu_t^{\omega^1}(\cdot)=P^{\omega^1}(X_t\in\cdot)$ for $t\in[0,T]$. Furthermore, if there exists an $\mathbb{F}$-progressively measurable $U$-valued process $\alpha=(\alpha_t)_{t\in [0,T]}$ on $\Omega$ such that $ {P^{\omega^1}}(\Lambda_t(\d u)\d t=\delta_{\alpha_t}(\d u)\d t)=1$, we say that $P^{\omega^1}$ corresponds to a strict control $\alpha$ or we call it a strict control rule. The set of all strict control rules is denoted by $\mathsf{R}^{\rm s}(\omega^1)$.
\end{definition}

We shall define the pathwise cost functional by, for any $\omega^1\in\Omega^1$,
\begin{align}\label{value_pathwise}
	\mathcal{J}(\omega^1,P):=\E^{P}\left[\int_0^T\int_Uf(t,X_t,\mu_t,u)\Lambda_t(\d u)\d t{+g(X_T,\mu_T)}\right],\quad \forall P\in\Pc_2(\Omega)
\end{align}
with $\mu_t:=P\circ X_t^{-1}$ for $t\in[0,T]$.
Introduce the set of optimal pathwise control rules defined by
\begin{align}\label{eq:RMopt}
	\mathsf{R}^{\rm opt}(\omega^1):=\mathop{\arg\min}\limits_{P^{\omega^1}\in \mathsf{R}(\omega^1)}\mathcal{J}(\omega^1,P^{\omega^1}).    
\end{align}

\begin{remark}
 %   Note that the above cost functional $\mathcal{J}(\omega^1,P)$ does not necessarily depend on $\omega^1$ when we take $P\in\Pc_2(\Omega)$, but we retain $\omega^1$ in the notation of $\mathcal{J}(\omega^1,P)$ to highlight its definition in the pathwise setting of \defref{relaxed_control_pathwise} for a fixed sample path $\omega^1\in\Omega^1$.

We stress that the measurability of $P^{\omega^1}$ with respect to $\omega^1\in\Omega^1$ is not required in the above definition. However, in the sequel, we will show the existence of a measurable selection of $\omega^1\mapsto \mathsf{R}^{\rm opt}(\omega^1)$, and hence the optimal value function $\inf_{P^{\omega^1}\in \mathsf{R}(\omega^1)}\mathcal{J}(\omega^1,P^{\omega^1})$ is measurable with respect to $\omega^1$.
\end{remark}

We then have the following martingale characterization and the corresponding moment estimate for the pathwise admissible relaxed control.
\begin{lemma}\label{moment_p_pathwise}
Let $\omega^1\in\Omega^1$ be fixed. Then,  $P^{\omega^1}\in \mathsf{R}(\omega^1)$ iff there exists a filtered probability space $(\Omega',\F',\Fb'=(\F'_t)_{t\in[0,T]},P')$ supporting a $\Pc(U)$-valued $\Fb'$-progressively measurable process $\Lambda^{\omega^1}=(\Lambda_t^{\omega^1})_{t\in [0,T]}$, an $\R^n$-valued $\Fb'$-adapted process $X^{\omega^1}=(X_t^{\omega^1})_{t\in[0,T]}$, a standard $n$-dimensional $(P',\Fb')$-Brownian motion $W^{\omega^1}=(W_t^{\omega^1})_{t\in [0,T]}$ and an $\R^n$-valued $\Fb'$-martingale measure ${\cal M}^{\omega^1}$ on $[0,T]\times U$, with intensity $\Lambda^{\omega^1}_t(\d u)\d t$  such that $P^{\omega^1}=P'\circ (X^{\omega^1},\Lambda^{\omega^1},W^{\omega^1})^{-1}$,  and it holds that {\rm(i)} $P'\circ(X_0^{\omega^1})^{-1}=\lambda$; {\rm(ii)} $W_t^{\omega^1}=\int_0^t\int_U\mathcal{M}^{\omega^1}(\d u,\d s)$, $\forall t\in [0,T]$, $P'$-$\as$; {\rm(iii)} the dynamics of state process $X^{\omega^1}$ obeys that, $P'$-a.s.,
\begin{align*}
\d X^{\omega^1}_t&=\int_Ub(t,X^{\omega^1}_t,\mu_t^{\omega^1},u)\Lambda^{\omega^1}_t(\d u)\d t+\int_U\sigma(t,X^{\omega^1}_t,\mu_t^{\omega^1},u){\cal M}^{\omega^1}(\d u,\d t)\\
&\quad+\int_Z\gamma(t,X^{\omega^1}_{t-},\mu_{t-}^{\omega^1},z)\omega^1(\d t,\d z)
\end{align*}    
with $\mu_t^{\omega^1}=\Law^{P'}(X_t^{\omega^1})$ for $t\in[0,T]$.   Moreover, there exists a constant $C > 0$ depending on $p,M$ and $M_p(\lambda)$, as well a constant $C_0>0$ only depending on $M$, such that
\begin{align}\label{eq:momentp00}
\E^{P'}\left[\sup_{t \in [0,T]} \left|X_t^{\omega^1}\right|^p\right] \leq C_0^{\left|D_{p^{\omega^1}}\right|+1} e^{CT},
\end{align}
where $\left|D_{p^{\omega^1}}\right|$ denotes the cardinality of the domain $D_{p^{\omega^1}}$.
\end{lemma}

\begin{proof}
	The proof of martingale measure characterization is standard and we only focus on the moment estimate. We fix $\omega^1\in\Omega^1$ and let $0=t_0^{\omega^1}<t_1^{\omega^1}<\cdots<t_k^{\omega^1}\leq t_{k+1}^{\omega^1}:=T$ be the jumping times under $\omega^1$ during $[0,T]$, i.e., the domain of definition of the corresponding point function $p^{\omega^1}$ is given by $D_{p^{\omega^1}}=\{t_1^{\omega^1},\ldots,t_k^{\omega^1}\}$. Here, $k$ ($k$ may depend on $\omega^1$, but we omit the superscript to ease the notation) is finite since the intensity measure $\nu(\d z)$ is finite.  Note that by standard moment estimation, we have, for $i=0,1,\dots,k$,
	\begin{align}\label{t_i}
		\E^{P'}\left[\sup_{t\in [t_i^{\omega^1},t_{i+1}^{\omega^1})}\left|X_t^{\omega^1}\right|^p\right]\leq e^{C\left(t_{i+1}^{\omega^1}-t_{i}^{\omega^1}\right)}\left\{1+\E^{P'}\left[\left|X_{t_i^{\omega^1}}^{\omega^1}\right|^p\right]\right\},
	\end{align}
	for some constant $C>0$ which depends on $p$ and $M$ only. We first consider $i=0$, $\ie$,
	\begin{align}\label{t_1}
		\E^{P'}\left[\sup_{t\in [0,t_1^{\omega^1})}\left|X_t^{\omega^1}\right|^p\right]\leq e^{Ct_1^{\omega^1}}\left\{1+\left(M_p(\lambda)\right)^p\right\}.
	\end{align}
	On the other hand, we have \begin{align}\label{Xt1}
		X_{t_1^{\omega^1}}^{\omega^1}=X_{t_1^{\omega^1}-}^{\omega^1}+\gamma\left(t_1^{\omega^1},X_{t_1^{\omega^1}-}^{\omega^1},\mu_{t_1^{\omega^1}-}^{\omega^1},p^{\omega^1}(t_1^{\omega^1})\right).
	\end{align}
	Therefore, by combining \equref{t_1} and \equref{Xt1} together, we can derive by using \assref{ass1}-(A4) that 
\begin{align}\label{t1_moment}
\E^{P'}\left[\left|X_{t_1^{\omega^1}}^{\omega^1}\right|^p\right]\leq (1+2M)e^{Ct_1^{\omega^1}}.
\end{align}
Here, the constant $C$ depends on $p,M,M_p(\lambda)$ and may be different from \equref{t_1} (and also may vary in the sequel).  Inserting \equref{t1_moment} into \equref{t_i} for $i=2$, we may derive similarly that $\E^{P'}[|X_{t_2^{\omega^1}}^{\omega^1}|^p]\leq (1+2M)^2e^{Ct_2^{\omega^1}}$.
%\begin{align*}
%	\E^{P'}\left[\left|X_{t_2^{\omega^1}}^{\omega^1}\right|^p\right]\leq (1+2M)^2e^{Ct_2^{\omega^1}}.
%	\end{align*}
By iterating this procedure, we obtain, for $i=1,\ldots,k$,
\begin{align}\label{Xti}		\E^{P'}\left[\left|X_{t_i^{\omega^1}}^{\omega^1}\right|^p\right]\leq (1+2M)^ie^{Ct_i^{\omega^1}}.
\end{align}
Combing \equref{t_i} and \equref{Xti}, we readily conclude the desired estimate \equref{eq:momentp00}.
\end{proof}

As a consequence of \lemref{moment_p_pathwise}, the set of admissible pathwise relaxed control $\mathsf{R}(\omega^1)$ is nonempty for every $\omega^1\in\Omega^1$. Moreover, thanks to \lemref{moment_p_pathwise}, we can provide an alternative characterization of $\mathsf{R}(\omega^1)$ without the proof in the next result.

\begin{lemma}\label{Y_characterization}
Let $\omega^1\in\Omega^1$ be fixed. We have $P^{\omega^1}\in \mathsf{R}(\omega^1)$ iff there exists an $\Fb$-adapted process $Y=(Y_t)_{t\in [0,T]}$ (depending on $P^{\omega^1}$) such that {\rm (i)} $Y$ is continuous with probability $1$; {\rm (ii)} $P^{\omega^1}\circ Y_0^{-1}=\lambda$; {\rm (iii)} $P^{\omega^1}(X_{\cdot}=Y_{\cdot}+\int_0^{\cdot}\int_Z\gamma(s,X_{s-},\mu_s^{\omega^1},z)\omega^1(\d s,\d z))=1$ with $\mu_t^{\omega^1}=P^{\omega^1}\circ X_t^{-1}$; {\rm (iv)} for any test function $\phi\in C_b^2(\R^n\times\R^n)$, the process
\begin{align*}
\tilde{\tt M}^{\omega^1,P^{\omega^1}}\phi(t):=\phi(Y_t,W_t)-\int_0^t\int_U\tilde{\Lb}\phi(s,X_s,Y_s,W_s,\mu_s^{\omega^1},u)\Lambda_s(\d u)\d s,\quad t\in[0,T]
\end{align*}
is a $(P^{\omega^1},\Fb)$-martingale. Here, the infinitesimal generator $\tilde{\Lb}$ acting on $\phi\in C_b^2(\R^n\times\R^n)$ is defined by, for $(t,x,y,\mu,u)\in [0,T]\times\R^n\times\R^n\times\Pc_2(\R^n)\times\R^l$,
\begin{align*}
\tilde{\Lb}\phi(t,x,y,w,\mu,u)(y)=\bar b(t,x,\mu,u)^{\T}\nabla\phi(y,w)+\frac12\tr\left(\bar\sigma\bar\sigma^{\T}(t,x,\mu,u)\nabla^2\phi(y,w)\right).
\end{align*}
\end{lemma}

\begin{remark}\label{Y_characterization_rem}
Motivated by \lemref{Y_characterization}, we can extend $\Omega$ to $\tilde\Omega:=\Omega\times\C$ and consider the product metric $d_{\tilde\Omega}(\tilde\omega_1,\tilde\omega_2):=d_{\Omega}(\omega_1,\omega_2)+\|\boldsymbol{y}_1-\boldsymbol{y}_2\|_{\infty}$ with $\tilde\omega_i=(\omega_i,\boldsymbol{y}_i)\in\tilde\Omega$, $i=1,2$ and the product $\sigma$-algebra $\tilde {\cal F}=\F\otimes\mathcal{B}(\C)$ with $\mathcal{B}(\C)$ being  the Borel $\sigma$-algebra of $\C$. Moreover, let $\F_t^{\C}$ be the Borel $\sigma$-algebra of $\C$ up to time $t$, and set $\tilde\Fb=(\tilde\F_t)_{t\in[0,T]}$ with $\tilde\F_t=\F_t\otimes\F_t^{\C}$. The coordinate mappings on $\tilde\Omega$ are defined by 
\begin{align*}
\tilde X_t(\tilde\omega)=\boldsymbol{x}(t),\quad \tilde\Lambda_t(\tilde\omega)=q_t,\quad \tilde W_t(\tilde\omega)=\boldsymbol{w}(t),\quad\tilde{Y}_t(\tilde\omega)=\boldsymbol{y}(t),\quad \forall \tilde\omega=(\boldsymbol{x},q,\boldsymbol{w},\boldsymbol{y})\in\tilde\Omega.
\end{align*}
Note that, if $P^{\omega^1}\in \mathsf{R}(\omega^1)$, then $P^{\omega^1}\circ (X,\Lambda,W,Y)^{-1}$ induces a probability measure on $(\tilde\Omega,\tilde\F)$, {where $(X,\Lambda,W,Y)$ under $P^{\omega^1}$ is the continuous process introduced in \lemref{Y_characterization}}. In this manner, we can restate \lemref{Y_characterization} as follows: $P^{\omega^1}\in \mathsf{R}(\omega^1)$ iff there exists a $\tilde{P}^{\omega^1}\in\Pc_2(\tilde{\Omega})$ such that {\rm (i)} $\tilde P^{\omega^1}(\tilde X_{\cdot}=\tilde Y_{\cdot}+\int_0^{\cdot}\int_Z\gamma(s,\tilde X_{s-},\tilde\mu_s^{\omega^1},z)\omega^1(\d z,\d s))=1$ with $\tilde{\mu}^{\omega^1}_t=\tilde{P}^{\omega^1}\circ \tilde{X}_t^{-1}$; {\rm(ii)} $\tilde P(\tilde W_0=0)=1$, $\tilde{P}\circ\tilde{Y}_0^{-1}=\lambda$ and $\tilde Y_0$ is $\tilde P$-independent of $\tilde W$; {\rm (iii)} for any test function $\phi\in C_b^2(\R^n\times\R^n)$, the following process
\begin{align*}
\tilde{\tt M}^{\omega^1,\tilde{P}^{\omega^1}}\phi(t):=\phi(\tilde Y_t,\tilde W_t)-\int_0^t\int_U\tilde{\Lb}\phi(s,\tilde X_s,\tilde Y_s,\tilde W_s,\tilde \mu_s^{\omega^1},u)(\tilde Y_s)\Lambda_s(\d u)\d s,\quad t\in[0,T]
\end{align*}
is a $(\tilde P^{\omega^1},\tilde\Fb)$-martingale.
\end{remark}
For $\bar{P}\in \mathsf{R}^{\rm s}$ in \defref{relaxed_control}, let us set $\rho_t(\omega^1)=\Law^{\bar{P}}((\bar{X}_t,\bar{\alpha}_t)|\F_t^1)(\omega^1)$ for $P^1$-$\as$ $\omega^1\in\Omega^1$. Then, the disintegration holds that $\rho_t(\omega^1)(\d x,\d u)=\hat{\alpha}_t(\omega^1)(x,\d u)\mu_t(\omega^1)(\d x)$.
As a result, for any test function $\phi\in C_b^2(\R^n)$, utilizing the martingality of ${\tt M}^{\bar P}\phi=({\tt M}^{\bar P}\phi(t))_{t\in[0,T]}$ under $\bar{P}$, it results in the following Fokker-Planck equation that, for all $t\in[0,T]$,
\begin{align}
\langle \phi,\mu_t\rangle&=\langle \phi,\lambda\rangle+\E^{\bar{P}}\left[\int_0^t\int_U\mathbb{L}\phi(s,\bar{X}_s,\mu_s,u)\bar{\Lambda}_s(\d u)\d s\middle|\F_t^1\right]\nonumber\\
&\quad+\E^{\bar{P}}\left[\int_0^t\int_Z(\phi(\bar{X}_{s-}+\gamma(s,\bar{X}_{s-},\mu_{s-},z))-\phi(\bar{X}_{s-}))\bar{N}(\d s,\d z)\middle|\F_t^1\right]\nonumber\\
&=\langle \phi,\lambda\rangle+\int_0^t\E^{\bar{P}}\left[\mathbb{L}\phi(s,\bar{X}_s,\mu_s,\bar{\alpha}_s)\middle|\F_t^1\right]\d s\nonumber\\
&\quad+\int_0^t\E^{\bar{P}}\left[\int_Z\left(\phi(\bar{X}_{s-}+\gamma(s,\bar{X}_{s-},\mu_{s-},z))-\phi(\bar{X}_{s-})\right)\bar{N}(\d s,\d z)\middle|\F_t^1\right]\nonumber\\
&=\langle \phi,\lambda\rangle+\int_0^t\left\langle \int_U\mathbb{L}\phi(s,\cdot,\mu_s,u)\hat{\alpha}_s(\cdot,\d u),\mu_s\right\rangle\d s\nonumber\\
&\quad+\int_0^t\int_Z\left\langle\phi(\cdot+\gamma(s,\cdot,\mu_{s-},z))-\phi(\cdot),\mu_{s-}\right\rangle\bar{N}(\d s,\d z).\label{FP_mu}
\end{align}
Here, we have used the compatible condition \eqref{compatible}.

In view of the Fokker-Planck equation \eqref{FP_mu}, it is natural for us to also consider the pathwise measure-valued control in a model without common noise.

\begin{definition}[Pathwise Measure-Valued Control (without common noise)]\label{measure_valued}
	Let $\omega^1\in\Omega^1$ be fixed.	We call a couple of a c\`{a}dl\`{a}g $\Pc_2(\R^n)$-valued measure flow $\boldsymbol{\mu}^{\omega^1}=(\mu_t^{\omega^1})_{t\in [0,T]}$ and a (measurable) kernel $\hat{\alpha}^{\omega^1}:[0,T]\times\R^n\to\Pc(U)$, denoted by $\hat{\alpha}^{\omega^1}_t(x,\d u)$, a pathwise  admissible measure-valued control (denoted by $(\boldsymbol{\mu}^{\omega^1},\hat{\alpha}^{\omega^1})\in \mathsf{R}_{\rm FP}(\omega^1)$) if it holds that {\rm(i)}  $ \mu_0^{\omega^1}=\lambda$;   {\rm (ii)} for any $\phi\in C_b^2(\R^n)$, $\boldsymbol{\mu}^{\omega^1}=(\mu_t^{\omega^1})_{t\in [0,T]}$ solves the following Fokker-Planck equation:
	\begin{align}\label{FP_measure}
		\langle\phi,\mu_t^{\omega^1}\rangle&=\langle\phi,\lambda\rangle+\int_0^t\left\langle\int_U\mathbb{L}\phi(s,\cdot,\mu_s^{\omega^1},u)\hat{\alpha}_s^{\omega^1}(\cdot,\d u),\mu_s^{\omega^1}\right\rangle\d s\nonumber\\
		&\quad+\int_0^t\int_Z\left\langle \phi(\cdot+\gamma(s,\cdot,\mu_{s-}^{\omega^1},z))-\phi(\cdot),\mu_{s-}^{\omega^1}\right\rangle\omega^1(\d z,\d s).
	\end{align}
\end{definition}

For every $\omega^1\in\Omega^1$ and $(\boldsymbol{\mu}^{\omega^1},\hat{\alpha}^{\omega^1})\in \mathsf{R}_{\rm FP}(\omega^1)$, the corresponding value function is then defined by
\begin{align}\label{value_measure}
\mathscr{J}(\omega^1,\boldsymbol{\mu}^{\omega^1},\hat{\alpha}^{\omega^1})&:=\int_0^T\int_{\R^n}\int_Uf(t,x,\mu_t^{\omega^1},u)\hat{\alpha}^{\omega^1}_t(x,\d u)\mu_t^{\omega^1}(\d x)\d t\nonumber\\
&\quad+{\int_{\R^n}g(x,\mu_T^{\omega^1})\mu_T^{\omega^1}(\d x)}.
\end{align}

\begin{remark}
In \defref{measure_valued}, we do not require the measurability of $(\boldsymbol{\mu}^{\omega^1}, \hat{\alpha}^{\omega^1})$ with respect to $\omega^1$, which clearly broadens the applicability of our approach.
The kernel $\hat{\alpha}^{\omega^1}$ introduced in \defref{measure_valued} will play a crucial role in our analysis. Given a probability measure $Q \in \mathcal{P}_2(\D)$, we can recover a probability measure $P \in \mathcal{P}_2(\D\times\mathcal{Q})$ via the push-forward mapping $P = Q \circ \Phi_{\omega}^{-1}$, where the mapping $\Phi_{\hat{\alpha}^{\omega^1}}:\D\to\D\times\mathcal{Q}$ is defined by
\begin{align}\label{recover} 
\Phi_{\hat{\alpha}^{\omega^1}}(\boldsymbol{x}) := \left( \boldsymbol{x},\hat{\alpha}_t^{\omega^1}(\boldsymbol{x}(t),\d u)\d t \right),\quad \forall \boldsymbol{x}\in\mathcal{D}.
\end{align}
\end{remark}

\begin{remark}
Note that, in \assref{ass1}, we require that the jump coefficient $\gamma(\cdot)$ in \equref{double_reflection} is uncontrolled. When {the jump coefficient} $\gamma(\cdot)$ depends on the control variable, the Fokker-Planck equation \equref{FP_mu} becomes that
\begin{align}
\langle \phi,\mu_t\rangle&=\langle \phi,\lambda\rangle+\int_0^t\left\langle \int_U\mathbb{L}\phi(s,\cdot,\mu_s,u)\hat{\alpha}_s(\cdot,\d u),\mu_s\right\rangle\d s\nonumber\\
&\quad+\int_0^t\int_Z\left\langle\int_U\phi(\cdot+\gamma(s,\cdot,\mu_{s-},u,z))\hat{\alpha}_{s-}(x,\d u)-\phi(\cdot),\mu_{s-}\right\rangle\bar{N}(\d s,\d z).\label{FP_controlled}
\end{align}
However, to establish a superposition principle in the pathwise formulation analogous to \thmref{thm:equivalence}-{\bf (ii)}, the martingale condition in \defref{relaxed_control} would need to be modified accordingly that the process 
\begin{align*}
&{\tt M}^{\bar{P}}\phi(t):=\phi(\bar{X}_t,\bar{W}_t)-\int_0^t\int_U\bar{\mathbb{L}}\phi(s,\bar{X}_s,\bar{W}_s,\bar\mu_s,u)\bar\Lambda_s(\d u)\d s\\
&\qquad-\int_0^t\int_Z\int_U\left(\phi(\bar{X}_{s-}+\gamma(s,\bar{X}_{s-},\bar\mu_{s-},u,z),W_s)-\phi(\bar{X}_{s-},W_s)\right)\bar\Lambda_{s-}(\d u)\bar{N}(\d s,\d z),~ t\in[0,T]
\end{align*}
is a $(\bar\Fb,\bar P)$-martingale, which in turn leads to the following Fokker-Planck equation:
\begin{align*}
\langle \phi,\mu_t\rangle&=\langle \phi,\lambda\rangle+\int_0^t\left\langle \int_U\mathbb{L}\phi(s,\cdot,\mu_s,u)\hat{\alpha}_s(\cdot,\d u),\mu_s\right\rangle\d s\nonumber\\
&\quad+\int_0^t\int_Z\left\langle\phi\left(\cdot+\int_U\gamma(s,\cdot,\mu_{s-},u,z)\hat{\alpha}_{s-}(x,\d u)\right)-\phi(\cdot),\mu_{s-}\right\rangle\bar{N}(\d s,\d z).
\end{align*}
However, this Fokker-Planck equation differs substantially from \equref{FP_controlled}, which causes a technical gap in showing some equivalence results in Section \ref{sec:Equivalence}. Therefore, in the present paper, we restrict our attention to the case where $\gamma$ is uncontrolled, and leave the controlled jump case for the future study.
\end{remark}

\subsection{Existence of pathwise optimal controls}

The aim of this subsection is to show that the set $\mathsf{R}^{\rm opt}(\omega^1)$ of optimal pathwise relaxed controls defined by \eqref{eq:RMopt} is nonempty for any $\omega^1\in\Omega^1$ by applying the compactification argument in the model with deterministic jumping times under the Skorokhod topology. This approach is classical and can be traced back to {EI Karoui et al.} \cite{Karoui} and Haussmann and Suo~\cite{Haussmann1}.
\begin{prop}\label{existence_pathwise}
For any $\omega^1\in\Omega^1$,  the set $\mathsf{R}^{\rm opt}(\omega^1)\neq\varnothing$ and is compact. Moreover, there exists a measurable selection given by
	\begin{align}\label{mea_sele}
		\omega^1\mapsto P^{\omega^1}_*\in \mathsf{R}^{\rm opt}(\omega^1).
	\end{align}
	As a result, the value function $\inf_{P^{\omega^1}\in \mathsf{R}(\omega^1)}\mathcal{J}(\omega^1,P^{\omega^1})$ is measurable with respect to $\omega^1$.
\end{prop}

To prove \propref{existence_pathwise}, we need the following auxiliary results:
\begin{lemma}\label{compactness}
	For any $\omega^1\in\Omega^1$, the set $\mathsf{R}(\omega^1)$ is a compact subset of $\Pc_2(\Omega)$.
\end{lemma}

\begin{proof}
To start with, define $\tilde{\mathsf{R}}(\omega^1):=\{\tilde P^{\omega^1}=P^{\omega^1}\circ (X,\Lambda,W,Y)^{-1};~P^{\omega^1}\in \mathsf{R}(\omega^1)\}$ (recall \remref{Y_characterization_rem}), and it only suffices to show that $\tilde{\mathsf{R}}(\omega^1)$ is a compact subset of $\Pc_2(\tilde\Omega)$. We first prove that $\tilde{\mathsf{R}}(\omega^1)$ is tight. In fact, {we recall that $(\tilde{X},\tilde{\Lambda},\tilde{W},\tilde{Y})=(\tilde{X}_t,\tilde{\Lambda}_t,\tilde{W}_t,\tilde{Y}_t)_{t\in[0,T]}$ are the coordinate processes introduced in \remref{Y_characterization_rem}. Then, it follows from} \lemref{moment_p_pathwise}  and \lemref{Y_characterization} that, there is a constant $C>0$ only depending on $(M,\lambda,T)$ such that
\begin{align*}
\E^{\tilde{P}^{\omega^1}}\left[\left|\tilde Y_t-\tilde Y_s\right|^p\right]\leq C|t-s|^{\frac{p}{2}},\quad\forall \tilde{P}^{\omega^1}\in \tilde{\mathsf{R}}(\omega^1).
\end{align*}
It follows from Kolmogorov's criterion that $\{\tilde{P}^{\omega^1}\circ \tilde{Y}^{-1};~\tilde{P}^{\omega^1}\in \tilde{\mathsf{R}}(\omega^1)\}$ is tight. 
%\begin{align*}
%\inf_{\tilde{P}^{\omega^1}\in \tilde{\mathsf{R}}(\omega^1)}\tilde{P}^{\omega^1}\circ \tilde{Y}^{-1}(K^{\epsilon})\geq 1-\epsilon.
%\end{align*} 

On the other hand, recall that the c\`{a}dl\`{a}g continuity modulus $w_{\delta}'(\cdot)$ is defined by, for $\boldsymbol{x}\in \D$,
	\begin{align*}
		w_{\delta}'(\boldsymbol{x}):=\inf\left\{\max_{i\leq r}\sup_{s,t\in  [s_{i-1},s_i)}|\boldsymbol{x}(t)-\boldsymbol{x}(s)|;~0=s_0<\cdots<s_r=T,~\inf_{i<r}(t_i-t_{i-1})\geq \delta,~{r\geq1}\right\}.
	\end{align*}
If we define that, for $\tilde{\omega}\in\tilde{\Omega}$,
\begin{align*}
Z_t(\tilde\omega):=\int_0^t\int_Z\gamma(s,\tilde X_{s-},\tilde\mu_{s-}^{\omega^1},z)\omega^1(\d s,\d z)=\sum_{t_i^{\omega^1}\leq t}\gamma(t_i^{\omega^1},\tilde{X}_{t_i^{\omega^1}},\mu_t^{\omega^1},p^{\omega^1}(t_i^{\omega^1}))
\end{align*}
with $(t_i^{\omega^1})_{i=1}^k$ being the jump times of $\omega^1\in\Omega^1$, then we have
\begin{align*}
w_{\delta}'(Z_{\cdot}(\tilde\omega))\leq\max_{i\leq k}\sup_{\substack{s,t\in [t_{i-1}^{\omega^1},t_i^{\omega^1})\\|t-s|\leq\delta}}|Z_t(\tilde\omega)-Z_s(\tilde\omega)|=0,\quad {\rm whenever}~\delta<\min_{i}|t_i^{\omega^1}-t_{i-1}^{\omega^1}|.
\end{align*} 
{Furthermore, in light of \eqref{eq:momentp00} and the growth condition in \assref{ass1}-(A4) imposed on the jump coefficient $\gamma(\cdot)$, we can easily derive that $\sup_{\tilde P^{\omega^1}\in\tilde{\mathsf{R}}(\omega^1)}\E^{\tilde P^{\omega^1}}[\sup_{t\in [0,T]}|Z_t|^2]<\infty$. This implies that  
\begin{align*}
\lim_{a\to\infty}\sup_{\tilde P^{\omega^1}\in\tilde{\mathsf{R}}(\omega^1)}\tilde P^{\omega^1}\left(\sup_{t\in [0,T]}|Z_t|>a\right)=0.
\end{align*}
By the tightness criterion for c\`{a}dl\`{a}g processes (c.f. Theorem 13.2 in \cite{Billingsley}), it follows that $\{\tilde{P}^{\omega^1}\circ Z^{-1};~\tilde{P}^{\omega^1}\in \tilde{\mathsf{R}}(\omega^1)\}$ is tight. Hence, the joint distribution $\{\tilde{P}^{\omega^1}\circ (\tilde Y,Z)^{-1};~\tilde{P}^{\omega^1}\in \tilde{\mathsf{R}}(\omega^1)\}$ is also tight. Then, there exists a compact set $K^{\epsilon}\subset \C\times\D$ such that $\inf_{\tilde P^{\omega^1}\in\tilde{\mathsf{R}}(\omega^1)}\tilde P^{\omega^1}((\tilde Y,Z)\in K^{\epsilon})\geq 1-\epsilon$. Thanks to the continuity of the mapping $h:\C\times\D\to\D$ defined by $(\boldsymbol{y},\boldsymbol{z})\mapsto \boldsymbol{y}+\boldsymbol{z}$, the image $h(K^{\epsilon})$ is compact in $\D$. It then holds by definition that
\begin{align*}
\tilde P^{\omega^1}\left(\tilde X\in h(K^{\epsilon})\right)\geq \tilde P^{\omega^1}\left((\tilde Y,Z)\in K^{\epsilon}\right)\geq 1-\epsilon, \quad\forall \tilde P^{\omega^1}\in\tilde{\mathsf{R}}(\omega^1),
\end{align*}
which gives the tightness of $\{\tilde{P}\circ \tilde{X}^{-1};~\tilde{P}^{\omega^1}\in \tilde{\mathsf{R}}(\omega^1)\}$.
}

Lastly, note that $\mathcal{Q}$ is compact, and hence $\{\tilde{P}^{\omega^1}\circ (\tilde{\Lambda},\tilde W)^{-1};~\tilde{P}^{\omega^1}\in \tilde{\mathsf{R}}(\omega^1)\}$ is also tight. The $p$-order moment estimate provided in \lemref{moment_p_pathwise} can upgrade this tightness to precompactness in $\Pc_2(\tilde\Omega)$ (c.f. Proposition 5.2 in Lacker \cite{Lacker2}). 
	
Now, we are left to check the closedness of $\tilde{\mathsf{R}}(\omega^1)$. To do it, let $\tilde P_n^{\omega^1}\in\tilde{\mathsf{R}}(\omega^1)$ with $\tilde P_n^{\omega^1}\to \tilde P^{\omega^1}$ in $\Pc_2(\tilde\Omega)$ as $n\to\infty$. Then, we need to verify that $\tilde P^{\omega^1}\in \tilde{\mathsf{R}}(\omega^1)$. We follow the argument used in the proof of Lemma 3.7 in \cite{BWY1} to verify the condition given by  \remref{Y_characterization_rem}-\textnormal{(i)}. Our first step is to show that the following set
\begin{align*}
\mathcal{E}^{\tilde P^{\omega^1}}:=\left\{ \tilde\omega \in \tilde\Omega;~\tilde{X}_{\cdot} = \tilde{Y}_{\cdot} + \int_0^{\cdot}\int_Z \gamma\left(s, \tilde{X}_{s-}, \tilde\mu_{s-}^{\omega^1},z\right) \, \omega^1(\d z,\d s) \right\}
\end{align*}
is closed in $ \tilde\Omega $ with $\tilde\mu_{t-}^{\omega^1}=\tilde P^{\omega^1}\circ \tilde{X}_{t-}^{-1}$. Assume that $\tilde\omega_n=(\boldsymbol{x}_n,q_n,\boldsymbol{w}_n,\boldsymbol{y}_n)\in\mathcal{E}^{\tilde P^{\omega^1}}$ converges to $\tilde\omega=(\boldsymbol{x},q,\boldsymbol{w},\boldsymbol{y})$ in $\tilde\Omega$ (with the product metric defined in \remref{Y_characterization_rem}) as $n\to\infty$, and we need to prove that $\tilde\omega\in\mathcal{E}^{\tilde P^{\omega^1}}$. In fact, we have from the definition that
\begin{align*}		\boldsymbol{x}_n(t)=\boldsymbol{y}_n(t)+\sum_{t_i^{\omega^1}\leq t}\gamma\left(t_i^{\omega^1},\boldsymbol{x}_n(t_i^{\omega^1}-),\mu_{t_i^{\omega^1}-},p^{\omega^1}(t_i^{\omega^1})\right),\quad\forall t\in [0,T].
\end{align*}
As a result, we deduce that $ \boldsymbol{x}_n(t) = \boldsymbol{y}_n(t) $ for all $ t \in [0, t_1^{\omega^1})$, and accordingly $\boldsymbol{x}_n(t_1^{\omega^1}-) = \boldsymbol{y}_n(t_1^{\omega^1})$ by using the continuity of $t\to \boldsymbol{y}_n(t)$. Since $(\boldsymbol{x}_n,\boldsymbol{y}_n)\to (\boldsymbol{x},\boldsymbol{y})$ in $\D\times\C$ as $n\to\infty$, it follows that
\begin{align*}
\lim_{n \to \infty} \boldsymbol{x}_n(t_1^{\omega^1}-) = \boldsymbol{y}(t_1^{\omega^1}),\quad \lim_{n\to\infty}\boldsymbol{x}_n(t)=\boldsymbol{y}(t),\quad \forall t\in [0,t_1^{\omega^1}).
\end{align*}

Proceeding by induction, we obtain that
\begin{align*}
\lim_{n \to \infty} \boldsymbol{x}_n(t_i^{\omega^1}-)& = \boldsymbol{y}(t_i^{\omega^1})+\sum_{j=1}^{i-1}\gamma \left(t_j^{\omega^1},\lim_{n\to\infty}\boldsymbol{x}_n(t_{j}^{\omega^1}-),\mu_{t_j^{\omega^1}}-,p^{\omega^1}(t_j^{\omega^1})\right),\\
\lim_{n\to\infty}\boldsymbol{x}_n(t)&=\boldsymbol{y}(t)+\sum_{j=1}^{i-1}\gamma \left(t_j^{\omega^1},\lim_{n\to\infty}\boldsymbol{x}_n(t_{j}^{\omega^1}-),\mu_{t_j^{\omega^1}}-,p^{\omega^1}(t_j^{\omega^1})\right),\quad \forall t\in [t_{i-1}^{\omega^1},t_{i}^{\omega^1})
\end{align*}
with the convention $\sum_{j=1}^0=0$. Moreover, one can easily verify by induction that the above convergence holds uniformly in $t$ as $(\boldsymbol{x}_n,\boldsymbol{y}_n)\to (\boldsymbol{x},\boldsymbol{y})$ in $\D\times\C$ as $n\to\infty$. Next, let $\boldsymbol{z}$ be the pointwise limit of $\boldsymbol{x}_n$ as $n\to\infty$, i.e. $\boldsymbol{z}(t):=\lim_{n\to\infty}\boldsymbol{x}_n(t)$ for $t\in[0,T]$. Note that $\lim_{n\to\infty}\boldsymbol{x}_n(T)=\lim_{n\to\infty}\boldsymbol{x}_n(T-)$ if $t_k^{\omega^1}<T$. Consequently, $(\boldsymbol{z},q,{\boldsymbol{w}},\boldsymbol{y})\in\mathcal{E}^{\tilde P^{\omega^1}}$. By construction, we also have $\|\boldsymbol{z}-\boldsymbol{x}_n\|_{\infty}\to 0$ as $n\to\infty$, which yields that $d_{\D}(\boldsymbol{x}_n,\boldsymbol{z})\to 0$ as $n\to\infty$. Hence, we derive $\boldsymbol{z}=\boldsymbol{x}$, and thus $\tilde\omega\in\mathcal{E}^{\tilde{P}^{\omega^1}}$. 
	
We next prove that 
\begin{align}\label{measure_convergence}
\limsup_{n\to\infty}\tilde P_n^{\omega^1}\left(\mathcal{E}^{\tilde{P}_n^{\omega^1}}\backslash\mathcal{E}^{\tilde{P}^{\omega^1}}\right)=0,
\end{align}
where the set $\mathcal{E}^{\tilde{P}_n^{\omega^1}}$ is defined by
\begin{align*}
\mathcal{E}^{{\tilde{P}_n^{\omega^1}}}:=\left\{ \tilde\omega \in \tilde\Omega;~ \tilde{X}_{\cdot} = \tilde{Y}_{\cdot} + \int_0^{\cdot}\int_Z \gamma\left(s, \tilde{X}_{s-}, \tilde\mu_{s-}^{\omega^1,n},z\right)\omega^1(\d s,\d z)\right\},\quad \tilde\mu_{t-}^{\omega^1,n}:=\tilde{P}_n^{\omega^1}\circ\tilde X_{t-}^{-1}.
\end{align*}
Consider $\tilde\omega\in\mathcal{E}^{\tilde{P}_n^{\omega^1}}\backslash\mathcal{E}^{\tilde{P}^{\omega^1}}$. Then, there exists some $t_0\in [0,T]$ such that
\begin{align}\label{delta_set}
\begin{cases}
\displaystyle \tilde X_{t_0}(\tilde\omega)\neq \tilde Y_{t_0}(\tilde\omega)+\sum_{t_i^{\omega^1}\leq t_0}\gamma\left(t_i^{\omega^1},\tilde X_{t_i^{\omega^1}-},\tilde\mu^{\omega^1}_{t_i^{\omega^1}-},p^{\omega^1}(t_i^{\omega^1})\right),\\[0.6em]
\displaystyle \tilde X_{t_0}(\tilde\omega)=\tilde Y_{t_0}(\tilde\omega)+\sum_{t_i^{\omega^1}\leq t_0}\gamma\left(t_i^{\omega^1},\tilde X_{t_i^{\omega^1}-},\tilde\mu^{\omega^1,n}_{t_i^{\omega^1}-},p^{\omega^1}(t_i^{\omega^1})\right).
\end{cases}
\end{align}
By using \lemref{t_i_convergence}, we have $\tilde\mu^{\omega^1,n}_{t_i^{\omega^1}-}\to\tilde\mu^{\omega^1}_{t_i^{\omega^1}-}$ in $\Pc_2(\R^n)$ as $n\to\infty$. We can thus conclude that $\equref{delta_set}$ can not hold for $n$ large enough since the uniform continuity of $\gamma(\cdot)$ in $\mu\in\R^n$ (\assref{ass1}-{\rm (A2)}). In other words, $\mathcal{E}^{\tilde{P}_n^{\omega^1}}\backslash\mathcal{E}^{\tilde{P}^{\omega^1}}$ is empty when $n$ is large enough, and hence \equref{measure_convergence} holds. Accordingly, we arrive at
\begin{align*}
1=\lim_{n\to\infty}\tilde{P}_n^{\omega^1}\left(\mathcal{E}^{\tilde P_n^{\omega^1}}\right)&=\limsup_{n\to\infty}\tilde{P}_n^{\omega^1}\left(\mathcal{E}^{\tilde{P}_n^{\omega^1}}\backslash\mathcal{E}^{\tilde{P}^{\omega^1}}\right)+\limsup_{n\to\infty}\tilde{P}_n^{\omega^1}\left(\mathcal{E}^{\tilde{P}^{\omega^1}}\right)\nonumber\\
&\leq0+\tilde{P}^{\omega^1}\left(\mathcal{E}^{\tilde{P}^{\omega^1}}\right)=\tilde{P}^{\omega^1}\left(\mathcal{E}^{\tilde{P}^{\omega^1}}\right),
\end{align*}
which verifies the validity of \remref{Y_characterization_rem}-(i). Here, in the 2nd equality, we used the fact that $\tilde{P}_n^{\omega^1}(\mathcal{E}^{\tilde{P}_n^{\omega^1}}\backslash\mathcal{E}^{\tilde{P}^{\omega^1}})+\tilde{P}_n^{\omega^1}(\mathcal{E}^{\tilde{P}^{\omega^1}})=\tilde{P}_n^{\omega^1}(\mathcal{E}^{\tilde{P}_n^{\omega^1}})$ since $\tilde{P}_n(\mathcal{E}^{\tilde{P}_n^{\omega^1}})=1$; while we applied Portmaneau Theorem in the 3rd inequality. 
	
The initial condition  in \remref{Y_characterization_rem}-(ii) is straightforward to verify. We now turn to establishing the martingality condition given in \remref{Y_characterization_rem}-(iii). Following the proof of Theorem 3.7 in {Haussmann and Lepeltier} \cite{Haussmann}, we can derive that, for any $t\in[0,T]$ and $\phi\in C_b^2(\R^n\times\R^n)$, $\tilde {\tt M}^{\omega^1,\tilde P^{\omega^1}}\phi(t)$ is continuous in $\tilde\omega\in\tilde\Omega$. Hence, for any $0\leq s<t<T$ and bounded $\tilde\F_s$-measurable r.v. $\tilde h$, it holds that 
\begin{align*}
\lim_{n\to\infty}\E^{\tilde P_n^{\omega^1}}\left[\left(\tilde {\tt M}^{\omega^1,\tilde P^{\omega^1}}\phi(t)-\tilde {\tt M}^{\omega^1,\tilde P^{\omega^1}}\phi(s)\right)\tilde h\right]=\E^{\tilde P^{\omega^1}}\left[\left(\tilde {\tt M}^{\omega^1,\tilde P^{\omega^1}}\phi(t)-\tilde {\tt M}^{\omega^1,\tilde P^{\omega^1}}\phi(s)\right)\tilde h\right],
\end{align*}
since ${\tt M}^{\omega^1,\tilde P^{\omega^1}}\phi(t)$ has at most quadratic growth due to \assref{ass1}-(A3) and $\tilde P_n^{\omega^1}\to\tilde P^{\omega^1}$ in $\Pc_2(\tilde\Omega)$ as $n\to\infty$.
	
On the other hand, thanks to the Lipschitz continuity of $(b,\sigma)$ in $\mu\in\R^n$ (c.f. \assref{ass1}-(A2)) and \lemref{measure_convergence_characterization}, we have
\begin{align*}
\lim_{n\to\infty}\sup_{(t,\tilde\omega)\in[0,T]\times\tilde\Omega}\left|\tilde {\tt M}^{\omega^1,\tilde P^{\omega^1}}\phi(t)-\tilde {\tt M}^{\omega^1,\tilde P_n^{\omega^1}}\phi(t)\right|=0.
\end{align*}
Lastly, we can conclude that
\begin{align*}
&\E^{\tilde P^{\omega^1}}\left[\left(\tilde {\tt M}^{\omega^1,\tilde P^{\omega^1}}\phi(t)-\tilde {\tt M}^{\omega^1,\tilde P^{\omega^1}}\phi(s)\right)\tilde h\right]=\lim_{n\to\infty}\E^{\tilde P_n^{\omega^1}}\left[\left(\tilde {\tt M}^{\omega^1,\tilde P^{\omega^1}}\phi(t)-\tilde {\tt M}^{\omega^1,\tilde P^{\omega^1}}\phi(s)\right)\tilde h\right]\\
&\quad=\lim_{n\to\infty}\E^{\tilde P_n^{\omega^1}}\left[\left(\tilde {\tt M}^{\omega^1,\tilde P_n^{\omega^1}}\phi(t)-\tilde {\tt M}^{\omega^1,\tilde P_n^{\omega^1}}\phi(s)\right)\tilde h\right]+\lim_{n\to\infty}\E^{\tilde P_n^{\omega^1}}\left[\left(\tilde {\tt M}^{\omega^1,\tilde P^{\omega^1}}\phi(t)-\tilde {\tt M}^{\omega^1,\tilde P_n^{\omega^1}}\phi(t)\right)\tilde h\right]\\
&\qquad+\lim_{n\to\infty}\E^{\tilde P_n^{\omega^1}}\left[\left(\tilde {\tt M}^{\omega^1,\tilde P^{\omega^1}}\phi(s)-\tilde {\tt M}^{\omega^1,\tilde P_n^{\omega^1}}\phi(s)\right)\tilde h\right]=0,
\end{align*}
where, in the last equality, we have used the martingal property of ${\tt M}^{\omega^1,\tilde P_n^{\omega^1}}$ under $\tilde P_n^{\omega^1}$. Putting all pieces together, we have established the desired compactness of $\tilde{\mathsf{R}}(\omega^1)$.
\end{proof}

\begin{lemma}\label{continuity}
	For any $\omega^1\in\Omega^1$,  the pathwise cost functional $\mathcal{J}(\omega^1,P)$ defined by \eqref{value_pathwise} is continuous in $P\in\Pc_2(\Omega)$. As a result, $\mathsf{R}^{\rm opt}(\omega^1)$ is a compact nonempty subset of $\Pc_2(\Omega)$.
\end{lemma}

\begin{proof}
Following the proof of Lemma 3.5 in Haussmann and Suo \cite{Haussmann1}, we can show that, as $\omega_n=(\boldsymbol{x}_n,q_n,\boldsymbol{w}_n)\to\omega=(\boldsymbol{x},q,\boldsymbol{w})$ in $\Omega$ under the metric $d_{\Omega}$ when $n\to\infty$, the following convergence holds that, for any c\`{a}dl\`{a}g measure flow $\boldsymbol{\mu}=(\mu_t)_{t\in[0,T]}\in D([0,T];\Pc_2(\R^n))$,
\begin{align}\label{f_convergence}
\lim_{n\to\infty}\int_0^T\int_Uf(t,\boldsymbol{x}_n(t),\mu_t,u)q_n(t,\d u)\d t&=\int_0^T\int_Uf(t,\boldsymbol{x}(t),\mu_t,u)q(t,\d u)\d t,\\
\lim_{n\to\infty}g(\boldsymbol{x}_n(T),\mu_T)&=g(\boldsymbol{x}(T),\mu_T).\label{g_convergence}
	\end{align}
Suppose that $P^n\to P$ in $\Pc_2(\Omega)$ as $n\to\infty$. It holds that
\begin{align*}
&\left|\mathcal{J}(\omega^1,P^n)-\mathcal{J}(\omega^1,P)\right|\nonumber\\
&\quad\leq\left|\E^{P}\left[\int_0^T\int_Uf(t,X_t,\mu_t,u)\Lambda_t(\d u)\d t\right]-\E^{P^n}\left[\int_0^T\int_Uf(t,X_t,\mu_t,u)\Lambda_t(\d u)\d t\right]\right|\\
&\qquad+\E^{P^n}\left[\int_0^T\int_U\left|f(t,X_t,\mu_t,u)-f(t,X_t,\mu_t^n,u)\right|\Lambda_t(\d u)\d t\right]\nonumber\\
&\qquad{+\left|\E^{P^n}[g(X_T,\mu_T)]-\E^P[g(X_T,\mu_T)]\right|+\E^{P^n}\left[\left|g(X_T,\mu_T)-g(X_T,\mu_T^n)\right|\right]}\\
&=:I_1^n+I_2^n+{I_3^n+I_4^n}
\end{align*}
with $\mu_t^n=P^n\circ X_t^{-1}$ and $\mu_t=P\circ X_t^{-1}$ for $t\in[0,T]$. Thanks to \equref{f_convergence}, \eqref{g_convergence}, at most quadratic growth of $\int_0^T\int_U f(t,X_t,\mu_t,u)\Lambda_t(\d u)\d t$ and $g(X_T,\mu_T)$ in $\omega \in \Omega$ from \assref{ass1}-(A3), we can conclude that $I_1^n,I_3^n \to 0$ as $n\to\infty$. 
    
On the other hand, using \assref{ass1}-(A3) again, we have 
\begin{align*}
I_2^n\leq M\E^{P_n}\left[\int_0^T\mathcal{W}_{2,\R^n}(\mu_t^n,\mu_t)^2\d t\right]=M\int_0^T\mathcal{W}_{2,\R^n}(\mu_t^n,\mu_t)^2\d t.
\end{align*}
The R.H.S. of the above result converges to $0$ as $n\to\infty$ by applying \lemref{measure_convergence_characterization} together with the assumption that $P^n\to\ P$ in $\Pc_2(\Omega)$ as $n\to\infty$. {For the term $I_4^n$, as $\delta(T) = T$ holds for every $\delta \in \Delta$ (see \eqref{Skorokhod_metric} for the definition of $\Delta$), we deduce that, for any $\delta\in\Delta$ and $\boldsymbol{x},\boldsymbol{y}\in\D$, $|\boldsymbol{x}(T) - \boldsymbol{y}(T)| = |\boldsymbol{x}(\delta(T)) - \boldsymbol{y}(T)| \leq \|\boldsymbol{x} \circ \delta - \boldsymbol{y}\|_{\infty}$. Taking the infimum over all $\delta \in \Delta$ on both sides of the this inequality, we obtain $|\boldsymbol{x}(T) - \boldsymbol{y}(T)| \leq d_{\mathcal{D}}(\boldsymbol{x}, \boldsymbol{y})$. By the continuous mapping theorem, one has $P^n\circ X_T^{-1}\to P\circ X_T^{-1}$ in $\Pc_2(\Omega)$ as $n\to\infty$, and hence $I_4^n\to 0$ as $n\to\infty$.}

So far, we have shown that, for any $\omega^1\in \Omega^1$, $P\to\mathcal{J}(\omega^1,P)$ is continuous in $\Pc_2(\Omega)$. Thus, it follows from \lemref{compactness} that $\mathsf{R}(\omega^1)$ is compact, and hence $\mathcal{J}(\omega^1,\lambda)$ admits a minimum $\mathsf{R}(\omega^1)$, which ensures that $\mathsf{R}^{\rm opt}(\omega^1)$ is nonempty. One can easily verify that  $\mathsf{R}^{\rm opt}(\omega^1)$ is a closed subset of $\mathsf{R}(\omega^1)$, and hence it is also compact. The proof is then complete.
\end{proof}

For a set valued mapping $\mathcal{K}:X\to 2^{Y}$ (the power set of $Y$), let us define its graph $\Gr(\mathcal{K})$ as \begin{align}\label{grpah}
	\Gr(\mathcal{K})=\{(x,y)\in X\times Y;~x\in X,~ y\in \mathcal{K}(x)\}.
\end{align}
{
We next provide the following technical lemma.
\begin{lemma}\label{measurable_mapping}
Let $X$ and $Y$ be two Polish spaces equipped with their corresponding $\sigma$-algebras $\F_X$ and $\F_Y$.  Let $R:X\rightrightarrows Y$ be a set-valued mapping with closed value such that its graph $\Gr(R)$ is $\F_X\otimes\F_Y$-measurable. Assume that $f:X\times Y\to\R$ is $\F_X\otimes\F_Y$-measurable and is continuous in $y\in Y$. Then, the graph of the minimizer correspondence $R^{\rm opt}:=\{y\in R(x);~f(x,y)=\inf_{y\in R(x)}f(x,y)\}$ is also $\F_X\otimes\F_Y$-measurable.
\end{lemma}

\begin{proof}
Note that $\Gr(R)$ is $\F_X\otimes\F_Y$-measurable, by the Castaing representation theorem (c.f. Theorem 14.5 in
Rockafellar and Wets \cite{Rockafellar}), there exists a countable family of measurable selections $\{h_n\}_{n\geq1}$ such that $h_n(x)\in R(x)$ for all $x\in X$ and $n\geq1$, and $R(x)=\overline{\{h_n(x);~n\geq1\}}$ for $x\in X$. From the continuity of $y\to f(x,y)$, it follows that $\inf_{y\in R(x)}f(x,y)=\inf_{n\geq1} f(x,h_n(x))$. In view that $f$ is jointly measurable and $h_n$ for $n\geq1$ are measurable, $x\mapsto f(x,h_n(x))$ is $\F_X$-measurable. As the infimum of a countable family of measurable functions, $x\mapsto \inf_{y\in R(x)}f(x,y)$ is therefore $\F_X$ measurable. Consequently, the graph $\Gr(R^{\mathrm{opt}})= \{(x,y)\in X\times Y;~f(x,y)-\inf_{y'\in R(x)} f(x,y') = 0\}$ is $\F_X\otimes\F_Y$-measurable as the pre-image of $\{0\}$ under a jointly measurable function.
\end{proof}

In the next result, we establish the joint measurability of the graph $\Gr (\mathsf{R}^{\rm opt})$ for the set $\mathsf{R}^{\rm opt}$ of optimal relaxed controls.

\begin{lemma}\label{measurable_selection}
The graph of the (compact) set valued mapping $\omega^1\to \mathsf{R}^{\rm opt}(\omega^1)$ is $\F^1\otimes\mathcal{B}(\Pc_2(\Omega))$-measurable.
\end{lemma}

\begin{proof}
To start with, fix $\phi\in C^2_b(\R^n\times\R^n)$, $s,t\in [0,T]$ with $s<t$ and bounded $\F_s$-measurable $\RV$ $h$. We will verify that  the functional $H^{s,t}_{\phi,h}:\Omega^1\times\Pc_2(\Omega)\to\R$ defined by $H^{s,t}_{\phi,h}(\omega^1,P):=\E^P[({\tt M}^{\omega^1,P}\phi(t)-{\tt M}^{\omega^1,P}\phi(s))h]$ is jointly continuous. Let $(\omega^{1,n},P^n)\to (\omega^1,P)$ in $\Omega^1\times\Pc_2(\Omega)$, as $n\to\infty$. Hence, for every $\epsilon>0$, there exists a compact set $K_{\epsilon}\subset \Omega$ and $\omega_0\in\Omega$ such that
\begin{align}\label{L2tight}
\sup_{n\geq1} P^n(\Omega\backslash K_{\epsilon})<\epsilon.
\end{align}
It holds by definition that 
{\small\begin{align*}
&\left|H^{s,t}_{\phi,h}(\omega^1,P)-H^{s,t}_{\phi,h}(\omega^{1,n},P^n)\right|\nonumber\\
&\quad\leq\left|\E^P\left[\phi(X_t,W_t)h\right]-\E^{P^n}\left[\phi(X_t,W_t)h\right]\right|+\left|\E^P\left[\phi(X_s,W_s)h\right]-\E^{P^n}\left[\phi(X_s,W_s)h\right]\right|\\
&\qquad+\E^{P^n}\left[\int_s^t\int_U\left|\bar{\mathbb{L}}\phi(r,X_r,W_r,\mu_r,u)\Lambda_r(\d u)-\bar{\mathbb{L}}\phi(r,X_r,W_r,\mu_r^n,u)\right|\left|h\right|\Lambda_r(\d u)\d r\right]\\
&\qquad+\left|\E^{P^n}\left[\int_s^t\int_U\bar{\mathbb{L}}\phi(r,X_r,W_r,\mu_r,u)h\Lambda_r(\d u)\d r\right]-\E^P\left[\int_s^t\int_U\bar{\mathbb{L}}\phi(r,X_r,W_r,\mu_r,u)h\Lambda_r(\d u)\d r\right]\right|\\
&\qquad+\E^{P^n}\left[\int_s^t\int_Z\left|\phi(X_{r-}+\gamma(r,X_{r-},\mu_{r-}^n,z),W_r)-\phi(X_{r-}+\gamma(r,X_{r-},\mu_{r-},z),W_r)\right|\left|h\right|\omega^{1,n}(\d r,\d z)\right]\\
&\qquad+\E^{P^n}\left[\left|\int_s^t\int_Z\left(\phi(X_{r-}+\gamma(r,X_{r-},\mu_{r-},z),W_r)-\phi(X_{r-},W_r)\right)h(\omega^{1,n}-\omega^1)(\d r,\d z)\right|\right]\\
&\qquad+\left|\left(\E^P-\E^{P^n}\right)\left[\int_s^t\int_Z\left(\phi(X_{r-}+\gamma(r,X_{r-},\mu_{r-},z),W_r)-\phi(X_{r-},W_r)\right)h\omega^1(\d r,\d z)\right]\right|\\
&\quad=\sum_{i=1}^7I_i,
\end{align*}}where $\mu_t^n:=P^n\circ X_t^{-1}$ and $\mu_t:=P\circ X_t^{-1}$ for $t\in[0,T]$.
By the convergence of $P^{1,n}\to P$ in $\Pc_2(\Omega)$, as $n\to\infty$, it follows that $I_1,I_2,I_4,I_7\to 0$ as $n\to\infty$. Thanks to Banach-Steinhaus theorem, $(\omega^{1,n})_{n\geq1}$ is uniformly bounded in $L^{\infty}([0,T]\times Z;\R)^*$, i.e.,  $L:=\sup_{n\geq1}\|\omega^{1,n}\|_{L^{\infty}([0,T]\times Z;\R)^*}<\infty$, and hence due to the Lipschitz continuity of $(b,\sigma,\gamma)$, it holds that, as $n\to\infty$,
\begin{align*}
I_3&\leq C\E^{P^n}\left[\int_s^t\left(\mathcal{W}_{2,\R^n}(\mu_{r-}^n,\mu_{r-})+\mathcal{W}_{2,\R^n}(\mu_{r-}^n,\mu_{r-})^2\right)|h|\d r\right]\to 0,\\
I_5&\leq C\E^{P^n}\left[\int_s^t\int_Z\mathcal{W}_{2,\R^n}(\mu_{r-},\mu_{r-})|h|\omega^{1,n}(\d r,\d z)\right]\to 0,
\end{align*}
for some constant $C>0$ which only depends on model coefficients, $h$ and $\phi$. For the term $I_6$, a simple calculation yields that
\begin{align*}        I_6&=\E^{P^n}\left[\left|\int_s^t\int_Z\left(\phi(X_{r-}+\gamma(r,X_{r-},\mu_{r-},z),W_r)-\phi(X_{r-},W_r)\right)h(\omega^{1,n}-\omega^1)(\d r,\d z)\right|\boldsymbol{1}_{K_{\epsilon}}\right]\\
&\quad+\E^{P^n}\left[\left|\int_s^t\int_Z\left(\phi(X_{r-}+\gamma(r,X_{r-},\mu_{r-},z),W_r)-\phi(X_{r-},W_r)\right)h(\omega^{1,n}-\omega^1)(\d r,\d z)\right|\boldsymbol{1}_{\Omega\backslash K_{\epsilon}}\right]\\
&=I_{6}^1+I_{6}^2.
\end{align*}
It is straightforward to verify that the mapping $\Gamma:\Omega\to L^{\infty}([0,T]\times Z)$ defined by $\Gamma(\omega)(t,z):=\phi(X_{t-}+\gamma(t,X_{t-},\mu_{t-},z),W_t)-\phi(X_{t-},W_t)$ is continuous, and hence the image $\Gamma(K_{\epsilon})$ is compact in $L^{\infty}([0,T]\times Z)$. Since $(\omega^{1,n})_{n\geq1}$ is equicontinuous and uniformly on $\Gamma(K_{\epsilon})$ by Banach-Steinhaus theorem, we conclude by Arzela-Ascoli theorem that the sequence of restrictions $(\omega^{1,n}|_{K_{\epsilon}})_{n\geq1}$ is compact in $C(K_{\epsilon};\R)$. Note that, for any $\psi\in K_{\epsilon}$, we have $\int_0^T\int_Z \psi(t,z)\omega^{1,n}(\d t,\d z)\to \int_0^T\int_Z\psi(t,z)\omega^1(\d t,\d z)$ as $n\to\infty$. Hence, every limit point of $\omega^{1,n}|_{K_{\epsilon}}$ as $n\to\infty$ in $C(K_{\epsilon};\R)$ must equal to $\omega^1|_{K_{\epsilon}}$, i.e., \begin{align}\label{compact_uniform}
\lim_{n\to\infty}\sup_{\omega\in\Omega}\left|\int_0^T\int_Z\Gamma(\omega)(t,z)\omega^{1,n}(\d t,\d z)-\int_0^T\int_Z\Gamma(\omega)(t,z)\omega^1(\d t,\d z)\right|=0.
\end{align}
Thanks to the boundedness of $h$, we further derive that
\begin{align}\label{compacth_uniform}
\lim_{n\to\infty}\sup_{\omega\in\Omega}\left[\left|\int_0^T\int_Z\Gamma(\omega)(t,z)\omega^{1,n}(\d t,\d z)-\int_0^T\int_Z\Gamma(\omega)(t,z)\omega^1(\d t,\d z)\right||h(\omega)|\right]=0,
\end{align}
which yields that $|I_{6}^1|< \epsilon$ when $n$ is large enough. In light of \eqref{L2tight}, we have $|I_{6}^2|< 2\|\phi h\|_{\infty}\epsilon$. By putting all previous estimates together, we obtain that $ \lim\limits_{n\to\infty}|H^{s,t}_{\phi,h}(\omega^1,P)-H^{s,t}_{\phi,h}(\omega^{1,n},P^n)|=0$. As the pre-image of $\{0\}$ under $H^{s,t}_{\phi,h}$, the set $\{H^{s,t}_{\phi,h}=0\}$ is thus closed in $\Omega^1\times\Pc_2(\Omega)$. On the other hand, we define a subset $\Pc_0$ of $\Pc_2(\Omega)$ as
\begin{align*}
\Pc_0:=\left\{P\in\Pc_2(\Omega);~P(W_0=0)=1,~P\circ X_0^{-1}=\lambda,~X_0~\text{is $P$-independent of}~W\right\}.
\end{align*}
One can easily show that $\Pc_0$ is closed in $\Pc_2(\Omega)$. Hence, by definition, we have the characterization of $\Gr(\mathsf{R})=(\cap_{s,t,\phi,h}\{H^{s,t}_{\phi,h}=0\})\cap(\Omega^1\cap\Pc_0)$, which yields that $\Gr(\mathsf{R})$ is closed (recall that the mapping $\mathsf{R}$ is defined in \defref{relaxed_control_pathwise}). Finally, the desired joint measurability of $\Gr(\mathsf{R}^{\rm opt})$ follows from Lemma~\ref{measurable_mapping}. 
One can easily show that $\Pc_0$ is closed in $\Pc_2(\Omega)$. Using its definition, we have the characterization of $\Gr(\mathsf{R})=(\cap_{s,t,\phi,h}\{H^{s,t}_{\phi,h}=0\})\cap(\Omega^1\cap\Pc_0)$, which implies that $\Gr(\mathsf{R})$ is closed. Finally, the desired joint measurability of $\Gr(\mathsf{R}^{\rm opt})$ follows from Lemma~\ref{measurable_mapping}. 
\end{proof}

Now, we are ready to prove \propref{existence_pathwise}:
\begin{proof}[Proof of \propref{existence_pathwise}]
The 1st assertion follows from \lemref{continuity}; while the 2nd assertion holds thanks to \lemref{measurable_selection} and Kuratowski–Ryll-Nardzewski measurable selection theorem (c.f. Corollary 14.6 in Rockafellar and Wets \cite{Rockafellar}).
\end{proof}}

\section{Step-2: Equivalence Between Different Formulations}\label{sec:Equivalence}

This section plays the key role in our pathwise formulation approach, which is devoted to establishing the equivalence between the original problem with common noise and the pathwise formulation when a sample path of common noise is fixed. To the best of our knowledge, these equivalence results are new to the existing literature.

The next theorem is the main result of this section.
\begin{theorem}\label{thm:equivalence}
	The following results on equivalence of formulations hold:
	\noindent{\rm (i)} In the original model with common noise, we have the equivalence between strict and relaxed control (in weak formulation) problems that
	\begin{align}\label{value_equivalence}
		\inf_{\bar{P}\in \mathsf{R}}\mathcal{J}(\bar{P})=	\inf_{\bar{P}\in \mathsf{R}^{\rm s}}\mathcal{J}(\bar{P}).
	\end{align}
	
	\noindent{\rm (ii)} (Superposition principle) In the pathwise formulation with a fixed $\omega^1\in\Omega^1$ and $(\boldsymbol{\mu}^{\omega^1},\hat{\alpha}^{\omega^1})\in \mathsf{R}_{\rm FP}(\omega^1)$, there exists a $P^{\omega^1}\in \mathsf{R}(\omega^1)$ such that, for $t\in [0,T]$,\begin{align}\label{consistency}
		P^{\omega^1}\circ X_t^{-1}=\mu_t^{\omega^1}(\d x),\quad P^{\omega^1}\left(\Lambda_{\cdot}=\hat{\alpha}_t^{\omega^1}(X_t,\d u)\d t\right)=1.
	\end{align}		
	Consequently, the following relationship holds that
	\begin{align}\label{pathwise_value_ineq}
	\inf_{(\boldsymbol{\mu}^{\omega^1},\hat{\alpha}^{\omega^1})\in \mathsf{R}_{\rm FP}(\omega^1)}\mathscr{J}(\omega^1,\boldsymbol{\mu}^{\omega^1},\hat{\alpha}^{\omega^1})\geq \inf_{P^{\omega^1}\in \mathsf{R}(\omega^1)}\mathcal{J}(\omega^1, P^{\omega^1}).
	\end{align}
	
	\noindent{\rm (iii)} We have the equivalence between the value function in \eqref{value_pathwise} in the pathwise formulation and the value function in the original model \eqref{value_common} in the following sense:
	\begin{align}\label{value_equivalent}
		\inf_{\bar{P}\in \mathsf{R}}\mathcal{J}(\bar{P})=\int_{\Omega^1}\inf_{P^{\omega^1}\in \mathsf{R}(\omega^1)}\mathcal{J}(\omega^1,P^{\omega^1})P^1(\d\omega^1).
	\end{align} 
	As a result, in the pathwise formulation, we have the equivalence that
	\begin{align}\label{pathwise_equivalence}
		\inf_{P^{\omega^1}\in {\mathsf{R}^{\rm s}(\omega^1)}}\mathcal{J}(\omega^1,P^{\omega^1})=\inf_{P^{\omega^1}\in \mathsf{R}(\omega^1)}\mathcal{J}(\omega^1,P^{\omega^1})=\inf_{(\boldsymbol{\mu}^{\omega^1},\hat{\alpha}^{\omega^1})\in \mathsf{R}_{\rm FP}(\omega^1)}\mathscr{J}(\omega^1,\boldsymbol{\mu}^{\omega^1},\hat{\alpha}^{\omega^1}).
	\end{align}
	Here, the second equality in \eqref{pathwise_equivalence} holds for $P^1$-$\as$ $\omega^1\in\Omega^1$; while the first equality in \eqref{pathwise_equivalence} holds for every $\omega^1\in\Omega^1$.
\end{theorem}

\begin{proof}%[Proof of \thmref{thm:equivalence}]
{\rm (i)} For any $\bar{P}\in \mathsf{R}$, let $\hat{P}\in\Pc_2(\hat\Omega)$ be the corresponding probability measure on $\hat\Omega$ (c.f.  \remref{relaxed_control_extension}). Then, we can obtain the existence of a sequence of $(\hat{P}_m)_{m\geq 1}\subset \hat{\mathsf{R}}^{\rm s}$ such that $\lim_{m\to\infty}\mathcal{W}_{2;\hat\Omega}(\hat{P}_m,\hat{P})=0$ by mimicking the proof of Proposition 7 and Lemma 4 in Djete et al. \cite{Djete}. {Indeed, note that the jump term $\gamma(\cdot)$ does not involve the control variable. Hence, when one applies the argument in the proof of Lemma 4 in \cite{Djete} to our setting, the difference of the jump term can be bounded by $M\E^{\hat{\Pb}_{\bar\omega}}[\int_0^t|\hat X_s^{\bar\omega,e}-\hat X_s|^2\d s]$ with $M$ defined in \assref{ass1}(we use the notation in \cite{Djete} here), which can be directly absorbed into the standard Gronwall's inequality and the rest of the arguments can follow the exact same steps as in \cite{Djete}.} On the other hand, if we set $\bar{P}_m=\hat{P}_m\circ (\hat X,\hat\Lambda,\hat W,\hat N)^{-1}$, one can easily check that $\bar{P}_m\in \mathsf{R}^{\rm s}$. Note that such push forward mapping is continuous, we also have that $\lim_{m\to\infty}\mathcal{W}_2(\bar{P}_m,\bar{P})=0$. By definition, it holds that
{\small\begin{align}\label{value_convergence}
\mathcal{J}(\bar{P})&=\E^{\bar P}\left[\int_0^T\int_Uf(t,\bar{X}_t,\bar{\mu}_t,u)\bar\Lambda_t(\d u)\d t{+g(\bar{X}_T,\bar{\mu}_T)}\right]\nonumber\\
&=\E^{\hat P}\left[\int_0^T\int_Uf(t,\hat{X}_t,\hat{\mu}_t,u)\hat\Lambda_t(\d u)\d t{+g(\hat X_T,\hat\mu_T)}\right]\nonumber\\
&=\lim_{m\to\infty}\E^{\hat P_m}\left[\int_0^T\int_Uf(t,\hat{X}_t,\hat{\mu}_t,u)\hat\Lambda_t(\d u)\d t{+g(\hat X_T,\hat\mu_T)}\right]\nonumber\\
&=\lim_{m\to\infty}\E^{\bar P_m}\left[\int_0^T\int_Uf(t,\bar{X}_t,\bar{\mu}_t^m,u)\bar\Lambda_t(\d u)\d t{+g(\bar X_T,\bar\mu_T)}\right]\nonumber\\
&=\lim_{m\to\infty}\mathcal{J}(\bar{P}_m),
\end{align}}where $\bar{\boldsymbol{\mu}}^m=(\bar\mu_t^m)_{t\in [0,T]}$ is the corresponding $\Fb^1$-adapted c\`{a}dl\`{a}g measure flow to $\bar{P}_m$. In view of \equref{value_convergence} and the arbitrariness of $\bar{P}\in \mathsf{R}$, we conclude that $\inf_{\bar{P}\in \mathsf{R}}\mathcal{J}(\bar{P})=\inf_{\bar{P}^{\rm s}\in \mathsf{R}^{\rm s}}\mathcal{J}(\bar{P}^{\rm s})$.
	
	{(ii)} Recall that the domain of definition of the corresponding point function $p^{\omega^1}$ is given by $D_{p^{\omega^1}}=\{t_1^{\omega^1},\ldots,t_k^{\omega^1}\}$. Let $(\boldsymbol{\mu}^{\omega^1},\hat{\alpha}^{\omega^1})\in \mathsf{R}_{\rm FP}(\omega^1)$ be a given pathwise measure-valued control. Then, the FP equation \equref{FP_measure} can be rewritten as, for $t\in[0,T]$,
	\begin{align}\label{FP_piecewise}
		\langle\phi,\mu_t^{\omega^1}\rangle&=\langle\phi,\lambda\rangle+\int_0^t\left\langle\int_U\mathbb{L}\phi(s,\cdot,\mu_s^{\omega^1},u)\hat{\alpha}_s^{\omega^1}(\cdot,\d u),\mu_s^{\omega^1}\right\rangle\d s\nonumber\\
		&\quad+\sum_{i=1}^k\left\langle\left(\phi(\cdot+\gamma(t_i^{\omega^1},\cdot,\mu_{t_i^{\omega^1}-},p^{\omega^1}(t_i^{\omega^1})))-\phi(\cdot)\right),\mu_{t_i^{\omega^1}-}\right\rangle {\bf1}_{\{t_i^{\omega^1}\leq t\}}.
	\end{align}
	In particular, $\mu_t^{\omega^1}$ for $t\in [0,t_1^{\omega^1})$ solves the following FP equation:
	\begin{align*}
		\langle\phi,\mu_t^{\omega^1}\rangle&=\langle\phi,\lambda\rangle+\int_0^t\left\langle\int_U\mathbb{L}\phi(s,\cdot,\mu_s^{\omega^1},u)\hat{\alpha}_s^{\omega^1}(\cdot,\d u),\mu_s^{\omega^1}\right\rangle\d s,\quad t\in [0,t_1^{\omega^1}).
	\end{align*}
	Thus, by applying the classical superposition principle (c.f. Theorem 2.5 in Trevisan \cite{Trevisan}), there exists a $Q^{\omega^1}_0\in\Pc_2(C[0,t_1^{\omega^1}],\R^n)$ such that $Q^{\omega^1}_0\circ \boldsymbol{x}(t)^{-1}=\mu_t^{\omega^1}$ for $t\in [0,t_1^{\omega^1})$, and for test function $\phi\in C_b^2(\mathbb{R}^n)$, it holds that
	\begin{align*}
		{\tt N}^{\boldsymbol{\mu}^{\omega^1}}\phi(t):=\phi(X_t)-\int_0^t\int_U\mathbb{L}\phi(s,X_s,\mu_s^{\omega^1},u)\Lambda_s(\d u)\d s,\quad t\in [0,t_1^{\omega^1}]
	\end{align*} 
	is a $(R_0^{\omega^1},\Fb^X\otimes\Fb^{\mathcal{Q}})$-martingale. Here,  $R_0^{\omega^1}:=Q_0^{\omega^1}\circ\Phi_{\hat{\alpha}^{\omega^1}}^{-1}$ (c.f. \equref{recover}, and in order to perform the push-forward mapping, we restrict $\Phi_{\hat{\alpha}^{\omega^1}}$ to the interval $[0,t_1^{\omega^1}]$). Similarly, we can construct $Q_1^{\omega^1},\dots,Q_{k}^{\omega^1}$ such that $Q_i^{\omega^1}\circ \boldsymbol{x}(t)^{-1}=\mu_t^{\omega^1}$ for $t\in [t_{i}^{\omega^1},t_{i+1}^{\omega^1})$, and $\{{\tt N}^{\boldsymbol{\mu}^{\omega^1}}\phi(t);~t\in [t_{i}^{\omega^1},t_{i+1}^{\omega^1}]\}$ is a $(R_i^{\omega^1},\Fb)$-martingale for $i=1,\ldots,k$, where $R_i^{\omega^1}=Q_i^{\omega^1}\circ\Phi_{\hat{\alpha}^{\omega^1}}^{\omega^1}$. Note that $\{{\tt N}^{\boldsymbol{\mu}^{\omega^1}}\phi(t);~t\in [t_1^{\omega^1},t_2^{\omega^1}]\}$ is a $(P_1^{\omega^1},\Fb^X\otimes\Fb^{\mathcal{Q}})$-martingale with initial law $\mu_{t_1^{\omega^1}}$. Hence, by applying Theorem 6.1.3 of Stroock and Varadhan \cite{Stroock}, we have,  for $\mu_{t_1}^{\omega^1}$-$\as$ $x\in \R^n$,  $\{{\tt N}^{\boldsymbol{\mu}^{\omega^1}}\phi(t);~t\in [t_1^{\omega^1},t_2^{\omega^1}]\}$ is a $(R_1^{\omega^1,x},\Fb^X\otimes\Fb^{\mathcal{Q}})$-martingale with initial value $x$, where $R_1^{\omega^1,x}=Q_1^{\omega^1,x}\circ \Phi_{\hat{\alpha}^{\omega^1}}^{-1}$ and $(Q_1^{\omega^1,x})_{x\in\R^n}$ is the r.c.p.d. of $Q_1^{\omega^1}$ given $\sigma(\boldsymbol{x}(t_2^{\omega^1}))$. 
	Note that, for $\phi\in C_b^2(\R^n)$, it holds that
	\begin{align*}
		\left\langle \phi,\mu_{t_1^{\omega^1}}\right\rangle=\left\langle\phi\left(\cdot+\gamma(t_1^{\omega^1},\cdot,\mu_{t_1^{\omega^1}-},p^{\omega^1}(t_1^{\omega^1}))\right),\mu_{t_1^{\omega^1}-}\right\rangle.
	\end{align*}
	Therefore, for $\mu_{t_1^{\omega^1}-}$-$\as$ $x\in\R^n$, there exists a family of probability measures $(Q_1^{\omega^1,x})_{x\in \R^n}\subset\Pc_2(C([t_1^{\omega^1},t_2^{\omega^1}];\R^n))$ that are measurable with respect to $x\in\R^n$ (still denoted by $Q_1^{\omega^1,x}$ for simplicity and the same for $R_1^{\omega^1,x}$ in the sequel) such that $\{{\tt N}^{\boldsymbol{\mu}^{\omega^1}}\phi(t);~t\in [t_1^{\omega^1},t_2^{\omega^1}]\}$ is a $(R_1^{\omega^1,x},\Fb^X\otimes\Fb^{\mathcal{Q}})$-martingale with initial value $x+\gamma(t_1^{\omega^1},x,\mu_{t_i^{\omega^1}-},p^{\omega^1}(t_1^{\omega^1}))$, where $R_1^{\omega^1,x}=Q_1^{\omega^1,x}\circ \Phi_{\hat{\alpha}^{\omega^1}}^{-1}$.
	
In view of \lemref{concatenation}, let us set $Q^{\omega^1}=Q_0^{\omega^1}\otimes_{t_1^{\omega^1}}Q_1^{\omega^1,\cdot}$. Thus, we have by construction (c.f. \lemref{concatenation}) that, for $A\in\mathcal{B}(\R^n)$ and $t\in [0,t_2^{\omega^1})$,
\begin{align*}
Q^{\omega^1}(\boldsymbol{x}(t)\in A)&=Q_0^{\omega^1}(\boldsymbol{x}(t)\in A){\bf1}_{\{t<t_1^{\omega^1}\}}+\E^{Q_0^{\omega^1}}\left[\delta^{\eta}\otimes_{t_1^{\omega^1}}Q_1^{\omega^1,\eta(t_1^{\omega^1})}(\boldsymbol{x}(t)\in A)\right]{\bf1}_{\{t\geq t_1^{\omega^1}\}}\\
&=\mu_t^{\omega^1}(A){\bf1}_{\{t<t_1^{\omega^1}\}}+Q_1^{\omega^1}(\boldsymbol{x}(t)\in A){\bf1}_{\{t\geq t_1^{\omega^1}\}}=\mu_t^{\omega^1}(A),
\end{align*}
where, in the penultimate equality, we used the tower property. As a result, for $t\in [0,t_2^{\omega^1})$, the consistency condition \equref{consistency} holds for $R^{\omega^1}$. {
Let us introduce the process that, for any test function $\phi\in C_b^2(\R^n)$,
\begin{align}\label{eq:Nw1Rw1}
{\tt N}^{\omega^1,R^{\omega^1}}\phi(t):&=\phi(X_t)-\int_0^t\int_U\Lb\phi(s,X_s,\mu_s^{\omega^1},u)\Lambda_s(\d u)\d s\nonumber\\
&\quad-\int_0^t\int_Z\left(\phi(X_{s-}+\gamma(s,X_{s-},\mu_s^{\omega^1},z)-\phi(X_{s-}) \right)\omega^1(\d s,\d z),~\forall t\in [0,T],
\end{align}
where $\mu_t^{\omega^1}=R^{\omega^1}\circ X_t^{-1}$ and the infinitesimal generator $\Lb$ acting on $\phi\in C_b^2(\R^n)$ is defined by
\begin{align*}
\Lb\phi(t,x,\mu,u):=b(t,x,\mu,u)^{\T}\nabla\phi(x)+\frac12\tr\left(\sigma\sigma^{\T}(t,x,\mu,u)\nabla^2\phi(x)\right).
\end{align*}}We next check that $\{{\tt N}^{\omega^1,R^{\omega^1}}\phi(t);~t\in [0,t_2^{\omega^1})\}$ is a $(R^{\omega^1},\Fb^X\otimes\Fb^{\mathcal{Q}})$-martingale with $R^{\omega^1}=Q^{\omega^1}\circ\Phi_{\hat{\alpha}^{\omega^1}}^{-1}$. Firstly, we have by definition that
	\begin{align*}
		R^{\omega^1}\left(\Lambda_t(\d u)=\hat{\alpha}^{\omega^1}(X_t,\d u),~\forall t\in [0,t_2^{\omega^1}]\right)=1.
	\end{align*}
	Thanks to the second assertion of \lemref{concatenation}, it only suffices to show that $\{{\tt N}^{\omega^1,R^{\omega^1}}\phi(t_1^{\omega^1}-\wedge t);~t\in [0,t_2^{\omega^1}]\}$ is a $(P_0^{\omega^1},\Fb^X\otimes\Fb^{\mathcal{Q}})$-martingale and $ \{{\tt N}^{\omega^1,R^{\omega^1}}\phi(t)-{\tt N}^{\omega^1,R^{\omega^1}}\phi(t_1^{\omega^1}-\wedge t);~t\in [0,t_2^{\omega^1}]\}$ is a $((\delta_{\eta}\otimes_{t_1^{\omega^1}}Q_1^{\omega^1,\eta(t_1^{\omega^1})})\circ\Phi_{\hat{\alpha}^{\omega^1}}^{-1},\Fb^X\otimes\Fb^{\mathcal{Q}})$-martingale for $Q_0^{\omega^1}$-$\as$ $\eta\in C([0,t_1^{\omega^1}];\R^n)$. Actually, the first martingale property follows from the construction of $Q_0^{\omega^1}$. To show the second martingale property, let us consider $0\leq s<t\leq t_2^{\omega^1}$. The martingale condition obviously holds when $t<t_1^{\omega^1}$ or $s\geq t_1^{\omega^1}$, and we only need to focus on the case $0\leq s<t_1^{\omega^1}\leq t\leq t_2^{\omega^1}$. Simple calculations yield that
    \begin{align*}
&\E^{((\delta_{\eta}\otimes_{t_1^{\omega^1}}Q_1^{\omega^1,\eta(t_1^{\omega^1})})\circ\Phi_{\hat{\alpha}^{\omega^1}}^{-1}}\left[{\tt N}^{\omega^1,R^{\omega^1}}\phi(t)-{\tt N}^{\omega^1,R^{\omega^1}}\phi(t_1^{\omega^1}-\wedge t)\middle|\F_s\right]\\
		&~=	\E^{((\delta_{\eta}\otimes_{t_1^{\omega^1}}Q_1^{\omega^1,\eta(t_1^{\omega^1})})\circ\Phi_{\hat{\alpha}^{\omega^1}}^{-1}}\left[{\tt N}^{\omega^1,R^{\omega^1}}\phi(t_1^{\omega^1})-{\tt N}^{\omega^1,R^{\omega^1}}\phi(t_1^{\omega^1}-)\middle|\F_s\right]\\
&~=\E^{((\delta_{\eta}\otimes_{t_1^{\omega^1}}Q_1^{\omega^1,\eta(t_1^{\omega^1})})\circ\Phi_{\hat{\alpha}^{\omega^1}}^{-1}}\left[\phi(X_t)-\phi\left(\eta(t_1^{\omega^1})+\gamma(t_1^{\omega^1},\eta(t_1^{\omega^1}),\mu_{t_1^{\omega^1}-},p^{\omega^1}(t_1^{\omega^1}))\right)\middle|\F_s\right]\\
		&~=0.
	\end{align*}
	Here, in the last inequality, we have used the fact that \begin{align*}
		Q_1^{\omega^1,\eta(t_1^{\omega^1})}\left(\boldsymbol{x}(t_1^{\omega^1})=\eta(t_1^{\omega^1})+\gamma(t_1^{\omega^1},\eta(t_1^{\omega^1}),\mu_{t_1^{\omega^1}-},p^{\omega^1}(t_1^{\omega^1}))\right)=1.
	\end{align*}
	We then proceed as in the case $t \in [0, t_2^{\omega^1})$ by applying the concatenation procedure to $Q^{\omega^1}$ iteratively to extend it to a probability measure in $\mathcal{P}_2(\D^n)$ (still denoted by $Q^{\omega^1}$ for simplicity). We finally define $R^{\omega^1}=Q^{\omega^1}\circ \Phi_{\hat{\alpha}^{\omega^1}}^{-1}$, which possesses the desired properties that can be verified in a similar manner. By \lemref{W_equivalence}, we conclude the existence of the desired probability measure $P^{\omega^1}\in \mathsf{R}(\omega^1)$. 
	
We next turn to the second assertion. By \thmref{thm:equivalence}-(ii), we have, for $(\omega^1,\boldsymbol{\mu}^{\omega^1},\hat{\alpha}^{\omega^1})\in \Omega^1\times \mathsf{R}_{\rm FP}(\omega^1)$, 
\begin{align}\label{value_preineq}
&\mathscr{J}(\omega^1,\boldsymbol{\mu}^{\omega^1},\hat{\alpha}^{\omega^1})=\int_0^T\int_Uf(t,x,\mu_t^{\omega^1},u)\hat{\alpha}^{\omega^1}_t(x,\d u)\mu_t^{\omega^1}(\d x)\d t{+\int_{\R^n}g(x,\mu_T^{\omega^1})\mu_T^{\omega_1}(dx)}\\
&\quad=	\E^{P^{\omega^1}}\left[\int_0^T\int_Uf(t,X_t,\mu_t^{\omega^1},u)\Lambda_t(\d u)\d t{+g(X_T,\mu_T^{\omega^1})}\right]=
\mathcal{J}(\omega^1,P^{\omega^1})\geq \inf_{Q\in \mathsf{R}(\omega^1)}\mathcal{J}(\omega^1,Q).\nonumber
\end{align}
By the arbitrariness of $(\boldsymbol{\mu}^{\omega^1},\hat{\alpha}^{\omega^1})$, we can conclude the claim in \equref{pathwise_value_ineq}.
	
	{\rm (iii)} On one hand, for any $\bar{P}\in \mathsf{R}^{\rm s}$, let us set
	\begin{align*}
		\hat{\alpha}_t^{\omega^1}(x,\d u)=\Law^{\bar{P}}(\bar{\alpha}_t|\F_t^1,\bar{X}_t=x)(\omega^1),\quad \mu_t^{\omega^1}=\Law^{\bar{P}}(\bar{X}_t|\F_t^1)(\omega^1),\quad \forall (t,\omega^1)\in[0,T]\times\Omega^1.
	\end{align*} 
Then, it holds that $(\boldsymbol{\mu}^{\omega^1}=(\mu_t^{\omega^1})_{t\in [0,T]},\hat{\alpha}^{\omega^1})\in \mathsf{R}_{\rm FP}(\omega^1)$ for $P^1$-$\as$ $\omega^1\in\Omega^1$ in lieu of \equref{FP_mu}.  Hence, for any $\bar{P}\in \mathsf{R}^{\rm s}$, we have  
\begin{align*}
\mathcal{J}(\bar{P})&=\E^{\bar P}\left[\int_0^T\int_Uf(t,\bar{X}_t,\mu_t,u)\bar\Lambda_t(\d u)\d t{+g(\bar{X}_T,\bar{\mu}_T)}\right]\\
&=\int_{\Omega^1}\left[\int_0^T\int_{\R^n\times U}f(t,x,\mu_t^{\omega^1},u)\hat{\alpha}^{\omega^1}_t(x,\d u)\mu_t^{\omega^1}(\d x)\d t{+\int_{\R^n}g(x,\mu_T^{\omega^1})\mu_T^{\omega^1}(\d x)}\right]P^1(\d\omega^1)\\
&=\int_{\Omega^1}	\mathscr{J}(\omega^1,\boldsymbol{\mu}^{\omega^1},\hat{\alpha}^{\omega^1})P^1(\d\omega^1)\geq\int_{\Omega^1}\inf_{P^{\omega^1}\in \mathsf{R}(\omega^1)}\mathcal{J}(\omega^1,P^{\omega^1})P^1(\d\omega^1),
\end{align*}
where we have used \equref{value_preineq} in the last equality. As a consequence, we obtain by the arbitrariness of $\bar{P}\in \mathsf{R}^{\rm s}$ that
\begin{align}\label{geq}
\inf_{\bar{P}\in \mathsf{R}^{\rm s}}\mathcal{J}(\bar{P})\geq\int_{\Omega^1}\inf_{P^{\omega^1}\in \mathsf{R}(\omega^1)}\mathcal{J}(\omega^1,P^{\omega^1})P^1(\d\omega^1).
\end{align}
On the other hand,  let $P^{\omega^1}_*$ be the measurable selection given in \equref{mea_sele}, and set 
\begin{align}\label{P_bar}
\bar{P}^*(\d\omega,\d\omega^1)=P^{\omega^1}_*(\d\omega)P^1(\d\omega^1).
\end{align} Our goal is to show that $\bar{P}^*\in\mathsf{R}$,  and hence the reverse inequality holds. We first identify the corresponding $\Fb^1$-adapted c\`{a}dl\`{a}g measure flow $\bar{\boldsymbol{\mu}}=(\bar\mu_t)_{t\in [0,T]}$. To this end, we first verify that $\mu_t^{\omega^1} = \Law^{\bar{P}^*}(\bar{X}_t | \F_t^1)$, $P^1$-a.s. 
Consider a measurable set of the form $B = B_1 \cap B_2$, where $B_1 \in \F_t^1$ and $B_2 = \left\{ \omega^1 \in \Omega^1;~\omega^1((t,s] \times A) \in F \right\}$ for some $A \in \mathscr{Z}$ and $F \in \mathcal{B}(\R_+)$. Then, for any $C \in \mathcal{B}(\R^n)$, it holds that
\begin{align*}
&\int_B \bar{P}^*(\bar{X}_t \in C | \F_t^1)(\omega^1) P^1(\d\omega^1)= \int_{B_1} \boldsymbol{1}_{B_2}(\omega^1) \bar{P}^*(\bar{X}_t \in C | \F_t^1)(\omega^1) P^1(\d\omega^1) \\
&\quad= P^1(B_2) \int_{B_1} \bar{P}^*(\bar{X}_t \in C | \F_t^1)(\omega^1) P^1(\d\omega^1)= P^1(B_2) \bar{P}^*(\bar{X}_t \in C, \bar{N} \in B_1) \\
&\quad= \bar{P}^*(\bar{X}_t \in C, \bar{N} \in B)= \int_B \mu_t^{\omega^1}(C) P^1(\d\omega^1).
\end{align*}
Here, the third and fifth equalities follow from the independence of $B_1$ and $B_2$ under $P^1$ (and hence under $\bar{P}^*$). Note that such measurable sets $B$ generate $\F^1$, the $\pi$-$\lambda$ theorem thus yields $\mu_t^{\omega^1} = \Law^{\bar{P}^*}(\bar{X}_t | \F_t^1)$, $P^1$-a.s.. Consequently, we define $\bar{\mu}_t(\bar{\omega}) := \mu_t^{\omega^1}$ for any $\bar{\omega} = (\omega, \omega^1)$, which verifies \defref{relaxed_control}-(iii).
	
	We next verify the martingale condition, because the rest conditions of \defref{relaxed_control} trivially hold.  Note that, for any $\phi\in C_b^2(\R^n\times\R^n)$, $0\leq s<t\leq T$ and $\bar{\F}_s$-measurable bounded $\RV$ $\bar{h}$, it follows from \defref{relaxed_control_pathwise}-(ii) that
{\footnotesize\begin{align}\label{equivalence_martingale}
		0&=\int_{\Omega^1}\E^{P^{\omega^1}_*}\left[\left(\mathtt{M}^{\omega^1,P^{\omega^1}_*}\phi(t)-\mathtt{M}^{\omega^1,P^{\omega^1}_*}\phi(t))\bar{h}(\cdot,\omega^1)\right)\right]P^1(\omega^1).\nonumber\\
		&=\int_{\Omega^1}\left(\int_{\Omega}\left(\phi(X_t(\omega),W_t(\omega))-\phi(X_s(\omega),W_s(\omega))-\int_s^t\int_U\mathbb{L}\phi(r,X_r(\omega),W_r(\omega),\mu_r^{\omega^1},u)\Lambda_r(\omega)\d r\right.\right.\nonumber\\
		&\left.\left.-\int_s^t\int_Z\left(\phi(X_{r-}(\omega)+\gamma(r,X_{r-}(\omega),\mu_{r-}^{\omega^1},z),W_r(\omega))-\phi(X_{r-},W_r(\omega))\right)\omega^1(\d r,\d z) \right)\bar{h}(\omega,\omega^1)P^{\omega^1}_*(\d\omega)\right)P^1(\d\omega^1)\nonumber\\
		&=\int_{\Omega^1}\int_{\Omega}\left(\phi(\bar{X}_t(\bar{\omega}),\bar{W}_t(\bar\omega))-\phi(\bar{X}_s(\bar{\omega}),\bar{W}_s(\bar\omega))-\int_s^t\int_U\mathbb{L}\phi(r,\bar{X}_r(\bar{\omega}),\bar{W}_r(\bar\omega),\bar\mu_r(\bar\omega),u)\bar{\Lambda}_r(\bar{\omega})\d r\right.\nonumber\\
		&\quad\left.-\int_s^t\left(\phi(\bar{X}_{r-}(\bar{\omega})+\gamma(r,\bar{X}_{r-}(\bar{\omega}),\bar\mu_{r-}(\bar\omega),z),\bar{W}_r(\omega))-\phi(\bar{X}_{r-}(\bar{\omega}),\bar{W}_r(\bar\omega))\right)\bar{N}(\bar{\omega})(\d r,\d z) \right)\bar{h}(\bar{\omega})\bar{P}^*(\d\bar{\omega})\nonumber\\
		&=\E^{\bar{P}^*}\left[\left(\mathtt{M}^{\bar{P}^*}\phi(t)-\mathtt{M}^{\bar{P}^*}\phi(s)\right)h\right],
	\end{align}}where, in the first equality, we have exploited the fact that $\bar{h}(\cdot,\omega^1)$ is $\F_s$-measurable for every $\omega^1\in\Omega^1$ and that $(\mathtt{M}^{\omega^1,P^{\omega^1}}\phi(t))_{t\in[0,T]}$ is a $(P_*^{\omega^1},\Fb)$-martingale for $\omega^1\in\Omega^1$. Therefore, we can conclude that $\bar{P}^*\in \mathsf{R}$ after validating (i)-(iii) of \defref{relaxed_control}. Finally, we can complete proof by definition that
\begin{align}\label{leq}
\inf_{\bar{P}\in \mathsf{R}}\mathcal{J}(\bar{P})\leq \mathcal{J}(\bar{P}^*)=\int_{\Omega^1}\mathcal{J}(\omega^1,P^{\omega^1}_*)P^1(\d\omega^1)=\int_{\Omega^1}\inf_{P^{\omega^1}\in \mathsf{R}(\omega^1)}\mathcal{J}(\omega^1,P^{\omega^1})P^1(\d\omega^1).
\end{align}
Combining \equref{value_equivalence}, \equref{geq} and \equref{leq}, we can readily deduce the equivalence \equref{value_equivalent}.  For the second assertion, the first equality of \equref{pathwise_equivalence} follows from a similar argument of item {\bf (i)} of \thmref{thm:equivalence} and the second equality holds in view of the definition, \equref{value_equivalent} and \equref{value_preineq} that
\begin{align*}
{\int_{\Omega^1}}\inf_{(\boldsymbol{\mu}^{\omega^1},\hat{\alpha}^{\omega^1})\in \mathsf{R}_{\rm FP}(\omega^1)}\mathscr{J}(\omega^1,\boldsymbol{\mu}^{\omega^1},\hat{\alpha}^{\omega^1}){P^1(\d\omega^1)}=\inf_{\bar{P}\in \mathsf{R}}\mathcal{J}(\bar{P})=\int_{\Omega^1}\inf_{P^{\omega^1}\in \mathsf{R}(\omega^1)}\mathcal{J}(\omega^1,P^{\omega^1})P^1(\d\omega^1).
\end{align*} 
Thus, we complete the proof of the theorem.
\end{proof}

\begin{remark}
\thmref{thm:equivalence}-(ii), new to the literature, can be interpreted as a superposition principle in the pathwise formulation with deterministic jumping times. Such formulation differs from the classical superposition result for continuous diffusion process (c.f. Theorem 2.5 in Trevisan \cite{Trevisan}) and the jump diffusion with L\'{e}vy jumps (c.f. Rockner et al. \cite{Rockner}). In particular, the infinitesimal generator associated with deterministic jumps involves Dirac-delta functions, which fall outside the analytical framework of \cite{Rockner}.
%	Meanwhile, although we have established several equivalence results in \thmref{thm:equivalence}, the relationship between the measure-valued control and the relaxed control in the original model (\defref{relaxed_control}) with Poissonian common noise, i.e., the superposition principle in the original model    remains a challenging open problem, which will be left for the future investigation.
\end{remark}

Finally, based on the preparations in the previous two-step procedure, we can now give the proof of the main result in \thmref{existence}. 
\begin{proof}[Proof of \thmref{existence}]
		The probability measure $\bar{P}^*(\d\omega,\d\omega^1)=P_*^{\omega^1}(\d\omega)P^1(\d\omega^1)$ defined in \equref{P_bar} belongs to $\mathsf{R}^{\rm opt}$ by construction and \equref{value_equivalent}. Consequently, $\mathsf{R}^{\rm opt}$ is nonempty.
\end{proof}

\begin{remark}\label{MFC_extension}

We note that the finite intensity of the Poisson random measure plays an important role to facilitate the pathwise formulation, as it ensures a well-defined pathwise construction of the stochastic integral with respect to the Poisson random measure. Moreover, the domain of the point function $p^{\omega^1}$ is finite, $\ie$, the set of jumping times $D_{p^{\omega^1}} = \{t_1^{\omega^1}, \ldots, t_k^{\omega^1}\}$ over the finite horizon contains only finitely many points. This differs substantially from the Brownian common noise, for which no analogous pathwise formulation is available and our pathwise formulation approach is not applicable.

%On the other hand, it is also important to highlight that 
%   our pathwise approach is not limited to Poissonian common noise only but can be adopted to handle other types of common noise. Essentially, if the common noise can be constructed pathwisely (for instance, if the common noise has finite variation), the pathwise formulation introduced in Subsection~\ref{pathwise_formulation} remains well-defined. Moreover, if a pathwise compactification can be performed to derive an optimal pathwise control as in \propref{existence_pathwise}, and a corresponding equivalence result between two formulations similar to \thmref{thm:equivalence} can be established, we can also obtain the existence of an optimal control for the original MFC in the presence of other types of common noise.
\end{remark}

\section{Auxiliary Results and Proofs}\label{auxiliary}
\subsection{Skorokhod topology}
For the sake of completeness, we present in this subsection some basic properties of the Skorokhod space $\D:=D([0,T];\R^n)$.

Let $\Delta$ be the collection of all time change functions, $\ie$ continuous strictly increasing functions $\delta:[0,T]\to [0,T]$ with $\delta(0)=0$ and $\delta(T)=T$. The Skorokhod metric $d_{\D}(\cdot,\cdot)$ is then defined by 
\begin{align}\label{Skorokhod_metric}
d_{\D}(\boldsymbol{x},\boldsymbol{y})=\inf_{\delta\in\Delta}\left\{\|\delta-I\|_{\infty}+\|\boldsymbol{x}-\boldsymbol{y}\circ \delta\|_{\infty}\right\},\quad \forall\boldsymbol{x},\boldsymbol{y}\in\D.
\end{align}
Here, $I:[0,T]\to[0,T]$ denotes the identity mapping on $[0,T]$ and $\boldsymbol{y}\circ\delta(t):=\boldsymbol{y}(\delta(t))$.
Then, $(\D,d_{\D})$ forms a Polish space.

\begin{lemma}\label{measure_convergence_characterization}
Let $P_n,P\in\Pc_2(\D)$ with $P_n\to P$ as $n\to\infty$ in $\Pc_2(\D)$. Then, it holds that
\begin{align*}
\lim_{n\to\infty}\int_0^T\mathcal{W}_{2,\R^n}\left(P_n\circ\boldsymbol{x}(t)^{-1},P\circ\boldsymbol{x}(t)^{-1}\right)^2=0.
\end{align*}
\end{lemma}

\begin{proof}
Thanks to Skorokhod representation theorem, there exists a probability space $(\Omega',\F',P')$ supporting a sequence of $\D^n$-valued r.v.s $X_n,X$ such that $P_n=\Law^{P'}(X_n)$, $P=\Law^{P'}(X)$ and $X_n\to X$ in $\D$ as $n\to\infty$, $P'$-$\as$. To be more precise, let ${\cal N}$ be a $P'$-null set such that $X_n(\omega')\to X(\omega')$ in $\D$ outside ${\cal N}$. For $\omega'\notin {\cal N}$, $X_n(\omega')$ is bounded in $\D$, and hence there exists $C>0$ independent of $n$ such that $d_D(X_n(\omega'),\boldsymbol{0})\leq C$, which yields $\|X_n(\omega')\|_{\infty}\leq C$ by using \equref{Skorokhod_metric}. On the other hand, $X_n(\omega')(t)$ converges to $X(\omega')(t)$ as $n\to\infty$ almost surely, and hence we have from by DCT that, for $\omega'\notin {\cal N}$,  $\int_0^T|X_n(\omega')(t)-X(\omega')(t)|^2\d t\to 0$ as $n\to\infty$. Furthermore, since $P_n\to P$ as $n\to\infty$ in $\Pc_2(\D)$, $(P_n)_{n\geq1}$ is uniformly bounded in $\Pc_2(\D)$, i.e., there exists a constant $C>0$ ($C$ may be different from $C$ above) independent of $n$ such that $\mathcal{W}_{2,\D}(P_n,\delta_{\boldsymbol{0}})\leq C$. This yields that $\sup_{n\geq1}\E^{P'}\left[\|X_n\|_{\infty}^2\right]\leq C$.
Hence, by Fubini's theorem and DCT again, we can finally conclude the desired result: 
\begin{align*}
\int_0^T\mathcal{W}_{2,\R^n}(P_n\circ\boldsymbol{x}(t)^{-1},P\circ\boldsymbol{x}(t)^{-1})^2\d t\leq \E\left[\int_0^T|X_n(t)-X(t)|^2\d t\right]\to 0,~n\to\infty.
\end{align*} 
Thus, we complete the proof of the lemma.
\end{proof}

\begin{lemma}\label{t_i_convergence}
	Let $P_n\to P$ in $\Pc_2(\Omega)$ as $n\to\infty$  with $(P_n)_{n\geq1}\subset \mathsf{R}(\omega^1)$. Then, for any $\omega^1\in\Omega^1$, $P_n\circ X_{t_i^{\omega^1}-}^{-1}\to P\circ X_{t_i^{\omega^1}-}^{-1}$ in $\Pc_2(\R^n)$ as $n\to\infty$ for $i=1,\ldots,k$, where the time sequence $(t_i^{\omega^1})_{i=1}^k$ is introduced in the proof of \thmref{thm:equivalence}-(ii).
\end{lemma}

\begin{proof}
Fix $\omega^1\in\Omega^1$, and recall the time sequence $(t_i^{\omega^1})_{i=1}^k$ introduced in the proof of \lemref{moment_p_pathwise} with $t_0^{\omega^1}=0$ and $t_{k+1}^{\omega^1}=T$. Let us define a subset of $\D^n$ as
\begin{align*}
\C^{\omega^1}:=\left\{
\boldsymbol{x}\in\D^n;~\boldsymbol{x}|_{[t_i^{\omega^1},t_{i+1}^{\omega^1})}\in C([t_i^{\omega^1},t_{i+1}^{\omega^1});\R^n),~i=0,1,\ldots,k\right\}.
	\end{align*}
We first show that $\C^{\omega^1}$ is closed. Let $\boldsymbol{x}_n\to\boldsymbol{x}$ in $\D$ as $n\to\infty$ with $(\boldsymbol{x}_n)_{n\geq1}\subset \C^{\omega^1}$. There exists a sequence $\delta_n\in\Delta$ such that $\|\boldsymbol{x}_n\circ\delta_n-\boldsymbol{x}\|_{\infty}+\|\delta_n-I\|_{\infty}\to 0$ as $n\to\infty$. Then, for any $t,s\in [t_i^{\omega^1},t_{i+1}^{\omega^1})$, we have $\delta^n(t),\delta^n(s)\in [t_i^{\omega^1},t_{i+1}^{\omega^1})$ for $n$ large enough. Furthermore, for any $\epsilon>0$, choose $n$ large enough such that $\|\boldsymbol{x}_n\circ\delta_n-\boldsymbol{x}\|_{\infty}<\epsilon/3$. Since $\boldsymbol{x}_n$ is continuous on $[t_i^{\omega^1},t_{i+1}^{\omega^1})$, there exists $\kappa>0$ such that $|\boldsymbol{x}_n(\delta_n(t))-\boldsymbol{x}(\delta_n(s))|<\epsilon/3$ when $|t-s|<\kappa$. Hence, we have $|\boldsymbol{x}(t)-\boldsymbol{x}(s)|\leq |\boldsymbol{x}_n(\delta_n(t))-\boldsymbol{x}(t)|+|\boldsymbol{x}_n(\delta_n(s))-\boldsymbol{x}(s)|+|\boldsymbol{x}_n(\delta_n(t))-\boldsymbol{x}_n(\delta_n(s))|\leq \epsilon$, whenever $|t-s|<\kappa$, which shows that $\boldsymbol{x}|_{(t_i^{\omega^1},t_{i+1}^{\omega^1})}\in C((t_i^{\omega^1},t_{i+1}^{\omega^1});\R^n)$. Note that $\boldsymbol{x}\in\D$, and hence is right continuous at $t_i^{\omega^1}$, which implies that $\boldsymbol{x}\in\C^{\omega^1}$ by the arbitrariness of $i$.
	
Note that $P_n\circ X^{-1}$ is supported on $\C^{\omega^1}$  by applying \lemref{moment_p_pathwise}. It follows from Portmaneau theorem that $P\circ X^{-1}(\C^{\omega})\geq\limsup_{n\to\infty}P_n\circ X^{-1}(\C^{\omega^1})=1$, which yields that $P\circ X^{-1}$ is also supported on $\C^{\omega^1}$.	Due to Skorokhod representation theorem, there exists a probability space $(\Omega',\F',P')$ supporting a sequence of $\D$-valued r.v.s $X_n',X'$ such that $P_n\circ X^{-1}=\Law^{P'}(X_n')$, $P\circ X^{-1}=\Law^{P'}(X')$ and $X_n'\to X'$ in $\D$, $P'$-$\as$.  Thanks to \lemref{moment_p_pathwise} again, there exists  a constant $C>0$ depending on $M,T$ such that
\begin{align}\label{continuous_ex}
\sup_{n\geq1}\E^{P'}\left[|X_n'(t)-X_n'(s)|^2\right]\leq C|t-s|.
\end{align}
Note that $X'(t)\to X'(t_i^{\omega^1}-)$ as $t\uparrow t_i^{\omega^1}$ $P'$-$\as$ and $\E^{P'}\left[\|X'\|_{\infty}\right]<\infty$ by following the same proof as in \lemref{measure_convergence_characterization}. We then conclude by DCT that
\begin{align}\label{X_convergence}
\lim_{t\uparrow t_i^{\omega^1}}\E^{P'}\left[\left|X'(t)-X'(t_i^{\omega^1}-)\right|^2\right]=0.
\end{align}
It holds by Cauchy's inequality that
\begin{align*}
&\mathcal{W}_{2,\R^n}\left(P_n\circ X_{t_i^{\omega^1}-}^{-1},P\circ X_{t_i^{\omega^1}-}^{-1}\right)\leq\E^{P'}\left[|X'(t_i^{\omega^1}-)-X_n'(t_i^{\omega^1}-)|^2\right]\leq 3\E^{P'}\left[|X_n'(t)-X_n'(t_i^{\omega^1}-)|^2\right]\\
&\qquad+3\E^{P'}\left[|X_n'(t)-X'(t)|^2\right]+3\E^{P'}\left[|X'(t_i^{\omega^1}-)-X'(t)|^2\right]=:I_1+I_2+I_3.
\end{align*}
In view of \eqref{X_convergence}, for any $\epsilon > 0$, there exists a $\kappa > 0$ such that $\E^{P'}[|X'(t) - X'(t_i^{\omega^1}-)|^2] < \frac{\epsilon}{3}$, whenever $t_i^{\omega^1} - t < \kappa $. We can further choose $\kappa $ small enough so that $t > t_{i-1}^{\omega^1} $, ensuring that $X' $ is continuous at $t$, and $t_i^{\omega^1} - t < \epsilon / (9C) $.  
	
Since $X_n' \to X' $ in $\D$, we have $X_n'(t) \to X'(t)$, as $n\to\infty$, $P' $-a.s. Then, by DCT (as in the proof of \lemref{measure_convergence_characterization}), we obtain $I_2 \to 0 $ as $n \to \infty $. Therefore, there exists $N > 0 $ such that $I_2 < \epsilon / 3 $ for all $n > N $.  As a result, we conclude that  $I_1 + I_2 + I_3 \leq 3C \cdot \frac{\epsilon}{9C} + \frac{\epsilon}{3} + \frac{\epsilon}{3} = \epsilon$,
	whenever $n > N$, where we have used \equref{continuous_ex}. 
\end{proof}

\subsection{Concatenation techniques}

This subsection is devoted to preparations for the technical proof of \thmref{thm:equivalence}-{\bf (ii)}, which relies on concatenation arguments. Our approach follows the methodology outlined in Section 6.1 of Stroock and Varadhan~\cite{Stroock} in which concatenation techniques are developed in the context of continuous diffusion. To start with, let $\boldsymbol{\mu}=(\mu_t)_{t\in [0,T]}$ be a c\`{a}dl\`{a}g measure flow and $p:D_p\to Z$ be a point function with a finite domain $D_p\subset [0,T]$. Fix $0\leq t_1<t_2<t_3\leq T$ such that $t_1,t_2\in D_p$, and define the following sets:
\begin{align*}
	\mathcal{X}_1&:=\{\boldsymbol{x}\in D([t_1,t_2];\R^n),\boldsymbol{x}(t_2)=\boldsymbol{x}(t_2-)\},\quad
	\mathcal{X}_2:=\{\boldsymbol{x}\in D([t_2,t_3];\R^n),\boldsymbol{x}(t_3)=\boldsymbol{x}(t_3-)\}.
\end{align*}
Then, we have
\begin{lemma}\label{pathwise_concatenation}
For any $\eta\in\mathcal{X}_1$, let $P^\eta(t_2)\in\Pc_2(\mathcal{X}_2)$ such that 
\begin{align*}
P^{\eta(t_2)}\left(\left\{\boldsymbol{x}(t_2)=\eta(t_2)+\gamma(t_2,\eta(t_2),\mu_{t_2-},p(t_2)) \right\}\right)=1.    
\end{align*}
Then, there exists a unique probability measure on $D([t_1,t_3];\R^n)$, denoted by $\delta_{\eta}\otimes_{t_2}P^{\eta(t_2)}$, such that $\delta_{\eta}\otimes_{t_2}P^{\eta(t_2)}(\boldsymbol{x}(t)=\eta(t),\forall t\in [t_1,t_2))=1$ and $\delta_{\eta}\otimes_{t_2}P^{\eta(t_2)}(A)=P^{\eta(t_2)}(A)$ for all $A\in \sigma(\boldsymbol{x}(t)$;~ $t\in [t_2,t_3])$.
\end{lemma}

\begin{proof}
	The uniqueness is trivial. For the existence, let us set
	\begin{align*}
		\mathcal{X}=\left\{(\boldsymbol{x}_1,\boldsymbol{x}_2)\in\mathcal{X}_1\times\mathcal{X}_2;~\boldsymbol{x}_2(t_2)=\boldsymbol{x}_1(t_2)+\gamma(t_2,\boldsymbol{x}_1(t_2),\mu_{t_2-},p(t_2)) \right\}.
	\end{align*}
	Then, $\mathcal{X}$ can be easily verified to be a measurable subset of $D([t_1,t_2];\R^n)\times D([t_2,t_3];\R^n)$. By Fubini theorem,  $\delta_{\eta}\otimes P^{\eta(t_2)}(\mathcal{X})=P^{\eta(t_2)}\left(\left\{\boldsymbol{x}(t_2)=\eta(t_2)+\gamma(t_2,\eta(t_2),\mu_{t_2-},p(t_2)) \right\}\right)=1$,
	%\begin{align*}
	%	\delta_{\eta}\otimes P^{\eta(t_2)}(\mathcal{X})=P^{\eta(t_2)}\left(\left\{\boldsymbol{x}(t_2)=\eta(t_2)+\gamma(t_2,\eta(t_2),\mu_{t_2-},p(t_2)) \right\}\right)=1,
	%\end{align*}
	where $\delta\otimes P^{\eta(t_2)}$ denotes the product measure of $\delta_{\eta}$ and $P^{\eta(t_2)}$. We then define the mapping $\Psi:\mathcal{X}\to D([t_1,t_3];\R^n)$ by\begin{align}\label{Psi}
		\Psi(\boldsymbol{x}_1,\boldsymbol{x}_2)=\boldsymbol{x}_1(t){\bf1}_{\{t_1\leq t<t_2\}}+\boldsymbol{x}_2(t){\bf1}_{\{t_2\leq t\leq t_3\}},\quad \forall (t,\boldsymbol{x}_1,\boldsymbol{x}_2)\in[t_1,t_3]\times\mathcal{X},
	\end{align}
	which is clearly measurable. Therefore, $(\delta\otimes P^{\eta(t_2)})\circ\Psi^{-1}$ is a probability measure on $D([t_1,t_3];\R^n)$ and it is easy to check that this is the desired probability measure $\delta\otimes_{t_2}P^{\eta(t_2)}$.
\end{proof}

\begin{lemma}\label{concatenation}
	Let $P_1\in\Pc_2(\mathcal{X}_1)$, and for $P_1\circ \boldsymbol{x}(t_2-)^{-1}$-$\as$ $x\in\R^n$, $x\to P^x$ be a measurable mapping from $\R^n$ to $\Pc_2(\mathcal{X}_2)$ such that $P^x\left(\left\{\boldsymbol{x}(t_2)=x+\gamma(t_2,x,\mu_{t_2-},p(t_2))\right\}\right)=1$. 
	Then, there exists a unique probability measure on $D([t_1,t_3];\R^n)$, denoted by $P_1\otimes_{t_2} P^{\cdot}$, such that $P_1\otimes_{t_2} P^{\cdot}$ equals $P_1$ on $\sigma(\boldsymbol{x}(t);~t\in [t_1,t_2))$ and $\delta_{\eta}\otimes_{t_2}P^{\eta(t_2)}$ is an r.c.p.d. of $P_1\otimes_{t_2} P^{\cdot}$ given $\sigma(\boldsymbol{x}(t);~t\in [t_1,t_2))$ for $P_1$-$\as$ $\eta\in\mathcal{X}_1$. In particular, suppose that $(\theta_t)_{t\in [t_1,t_3]}$ is an $\Fb$-progressively measurable c\`{a}dl\`{a}g process such that $\theta(t)$ is $P_1 \otimes_{t_2} P^{\cdot}$-integrable, $(\theta(t_2-\wedge t))_{t\in [t_1,t_3]}$ is a $P_1$-martingale  and $(\theta(t)-\theta(t_2-\wedge t))_{t\in [t_1,t_3]}$ is a $\delta_{\eta}\otimes P^{\eta(t_2)}$-martingale for $P_1$-$\as$ $\eta\in\mathcal{X}_1$, where $t_2-\wedge t:=t{\bf1}_{\{t<t_2\}}+t_2{\bf1}_{\{t\geq t_2\}}$. Then $(\theta(t))_{t\in[t_1,t_3]}$ is a $P_1\otimes_{t_2} P^{\cdot}$-martingale.
\end{lemma}
\begin{proof}
	To prove the first assertion, it suffices to verify that the mapping  
	\begin{align}\label{measurable_concatenation}
		\eta \mapsto \delta_{\eta} \otimes_{t_2} P^{\eta(t_2)},\quad\forall \eta\in\mathcal{X}_1
	\end{align}
	is measurable with respect to the $\sigma$-algebra $\sigma(\boldsymbol{x}(t);~t \in [t_1, t_2))$. Once done, we can define $P_1 \otimes_{t_2} P^{\cdot} := \mathbb{E}^{P_1} \left[ \delta_{\eta} \otimes_{t_2} P^{\eta(t_2)} \right]$,  
	which gives the desired probability measure. Let $A:=\{\boldsymbol{x}(s_1)\in\Gamma_1,\dots,\boldsymbol{x}(s_m)\in\Gamma_m\}$ with $m\geq1$, $t_1\leq s_1<\dots<s_j<t_2\leq s_{j+1}<\dots<s_m\leq t_3$ and $\Gamma_1,\dots,\Gamma_m\in\mathcal{B}(\R^n)$. Then, it holds that
	\begin{align}\label{decomposition}
		\delta_{\eta} \otimes_{t_2} P^{\eta(t_2)}(A)={\bf1}_{\Gamma_1}(\eta(s_1))\cdots{\bf1}_{\Gamma_j}(\eta(s_j))P^{\eta(t_2)}(\boldsymbol{x}(s_{j+1})\in\Gamma_{j+1},\dots,\boldsymbol{x}(s_m)\in\Gamma_m).
	\end{align}
	Note that, for $\eta\in\mathcal{X}_1$, the mapping $\eta\mapsto \eta(t_2)=\eta(t_2-)$ is $\sigma(\boldsymbol{x}(t);~t\in [t_1,t_2))$-measurable by construction. Hence, the measurability of the mapping \equref{measurable_concatenation} follows immediately from the measurability of the mapping $x\mapsto P^x$.
	
	For the second assertion, let $t_1\leq s<t\leq t_3$ and $A\in\sigma(\boldsymbol{x}(s);~t_1\leq r\leq s)$ be given. It holds that
	\begin{align*}
		\E^{P_1\otimes_{t_2} P^{\cdot}}\left[\theta(t){\bf1}_A\right]&=\E^{P_1\otimes_{t_2} P^{\cdot}}\left[\E^{\delta_{\eta} \otimes_{t_2} P^{\eta(t_2)}}\left[\theta(t){\bf1}_A\right]\right]=\E^{P_1\otimes_{t_2} P^{\cdot}}\left[\E^{\delta_{\eta} \otimes_{t_2} P^{\eta(t_2)}}\left[\theta((t_2-\wedge t)\vee s){\bf1}_A\right]\right]\\
		&=\E^{P_1\otimes_{t_2} P^{\cdot}}\left[\theta(s){\bf1}_A{\bf1}_{\{t_2\leq s\}}\right]+\E^{P_1\otimes_{t_2} P^{\cdot}}\left[\E^{\delta_{\eta} \otimes_{t_2} P^{\eta(t_2)}}\left[\theta(t_2-\wedge t){\bf1}_A{\bf1}_{\{s<t_2\}}\right]\right]\\
		&=\E^{P_1\otimes_{t_2} P^{\cdot}}\left[\theta(s){\bf1}_A{\bf1}_{\{t_2\leq s\}}\right]+\E^{P_1\otimes_{t_2} P^{\cdot}}\left[\theta(s){\bf1}_A{\bf1}_{\{s<t_2\}}\right]=\E^{P_1\otimes_{t_2} P^{\cdot}}\left[\theta(s){\bf1}_A\right],
	\end{align*}
	where we have utilized the martingale property of $\theta(t)-\theta(t_2-\wedge t)$ for $t\in[t_1,t_3]$ in the second equality and the martingale property of $\theta(t_2- \wedge t)$ for $t\in[t_1,t_3]$ in the penultimate equality. The proof is thus complete.
\end{proof}
\subsection{Equivalent formulation of \defref{relaxed_control_pathwise}}
Thanks to the martingale measure driven SDE representation, we have the following equivalent chracterization for $\mathsf{R}(\omega^1)$.
\begin{lemma}\label{W_equivalence}
Let $\omega^1\in\Omega^1$ be fixed. A probability measure $P^{\omega^1}$ belongs to $\mathsf{R}(\omega^1)$ iff there exists $R^{\omega^1}\in \Pc_2(\D\times\mathcal{Q})$ with $R^{\omega^1}=P^{\omega^1}\circ (X,\Lambda)^{-1}$ such that {\rm (i)} $R^{\omega^1}\circ X_0^{-1}=\lambda$; {\rm (ii)} for any test function $\phi\in C_b^2(\R^n)$, the process $\{{\tt N}^{\omega^1,R^{\omega^1}}\phi(t);~t\in[0,T]\}$ defined by \eqref{eq:Nw1Rw1}
%\begin{align}\label{eq:Nw1Rw1}
%{\tt N}^{\omega^1,R^{\omega^1}}\phi(t):&=\phi(X_t)-\int_0^t\int_U\Lb\phi(s,X_s,\mu_s^{\omega^1},u)\Lambda_s(\d u)\d s\\
%&\quad-\int_0^t\int_Z\left(\phi(X_{s-}+\gamma(s,X_{s-},\mu_s^{\omega^1},z)-\phi(X_{s-}) \right)\omega^1(\d s,\d z),~t\in [0,T]
%\end{align}
is a $(R^{\omega^1},\Fb^X\otimes\Fb^{\mathcal{Q}})$-martingale. % and the infinitesimal generator $\Lb$ acting on $\phi\in C_b^2(\R^n)$ is defined by
%\begin{align*}
%\Lb\phi(t,x,\mu,u)=b(t,x,\mu,u)^{\T}\nabla\phi(x)+\frac12\tr\left(\sigma\sigma^{\T}(t,x,\mu,u)\nabla^2\phi(x)\right).
%\end{align*}
\end{lemma}

\vspace{0.4cm}
{\noindent\textbf{Acknowledgements.} L. Bo and J. Wang are supported by  NNSF of China (No.12471451), Natural Science Basic Research Program of Shaanxi (No.2023-JC-JQ-05) and Shaanxi Fundamental Science Research Project for Mathematics and Physics (No.23JSZ010). X. Wei is supported by NNSF of China (No.12201343). X. Yu is supported by the Hong Kong RGC General Research Fund (GRF) (No.15306523 and No.15214125).}

\ \\

\end{document}